\documentclass[11pt,letterpaper,leqno]{article}
\usepackage[english]{babel}
\usepackage[utf8]{inputenc}
\usepackage{amsfonts,amsmath,amssymb,amsthm,pdfsync,mathrsfs}    
\usepackage{multirow}
\usepackage{enumerate}
\usepackage{comment}
\usepackage{cite}
\usepackage{subcaption} 
\usepackage{pict2e}
\usepackage{float}
\usepackage{bbm}
\usepackage{color}
\allowdisplaybreaks
\usepackage{graphicx}

\newtheorem{theorem}{Theorem}
\newtheorem{proposition}{Proposition}
\newtheorem{lemma}{Lemma}
\newtheorem{corollary}{Corollary}
\theoremstyle{definition}
\newtheorem{definition}{Definition}
\newtheorem{remark}{Remark}
\newtheorem{example}{Example}

\numberwithin{equation}{section}
\numberwithin{theorem}{section}
\numberwithin{lemma}{section}
\numberwithin{corollary}{section}
\numberwithin{proposition}{section}
\numberwithin{definition}{section}
\numberwithin{example}{section}
\numberwithin{remark}{section}

\newcommand{\suu}{\textnormal{\bf{supp}}}
\newcommand{\su}{\textnormal{\bf{supp}}^+}
\newcommand{\sm}{{\textnormal{\bf{Tsupp}}(m)}}
\newcommand{\smq}{{\textnormal{\bf{Tsupp}}(m,q)}}
\newcommand{\smqest}{{\widehat{\textnormal{\bf{Tsupp}}}(k,m,q)}}
\newcommand{\smest}{{\widehat{\textnormal{\bf{Tsupp}}}(k,m)}}
\newcommand{\s}{\mathbb{S}}
\newcommand{\st}{\mathbb{S}_T}

\newcommand{\prob}{\mathbb{P}}
\newcommand{\ex}{\mathbb{E}}
\newcommand{\bx}{\boldsymbol{x}}
\newcommand{\by}{\boldsymbol{y}}
\newcommand{\bz}{\boldsymbol{z}}
\newcommand{\Rast}{\text{Rast}}
\newcommand{\bTheta}{\boldsymbol{\Theta}}
\newcommand{\bU}{\boldsymbol{U}}

\newcommand{\Mxm}{\textnormal{M}({\bf x},m)}
\newcommand{\pa}{\textnormal{proj}_{{\bf a}_1,{\bf a}_2}}

\newcommand{\bone}{\boldsymbol 1}

\date{}

\title{Asymptotic independence and support detection techniques for
  heavy-tailed multivariate data} 
\author{
 Jaakko  Lehtomaa\footnote{Jaakko Lehtomaa was funded by the Finnish
Cultural Foundation.}\\
  \and
Sidney I.  Resnick\footnote{Sidney Resnick was partially supported by US ARO MURI grant W911NF-12-1-0385.}\\
}
\begin{document}
\maketitle

\begin{abstract}
One of the central objectives of modern risk management is to find a
set of risks where the probability of multiple simultaneous
catastrophic events is negligible. That is, risks are taken only when
their joint behavior seems sufficiently independent. This paper aims
to help to identify asymptotically independent risks by providing
additional tools for describing dependence structures of multiple
risks when the individual risks can obtain very large values. 

The study is performed in the setting of multivariate regular
variation. We show how asymptotic independence is connected to
properties of the support of the angular measure and present
an asymptotically consistent estimator of the support. The estimator
generalizes to any dimension $N\geq 2$ and requires no prior knowledge
of the support. The validity of the support estimate can be rigorously
tested under mild assumptions by an asymptotically normal test
statistic. 
\end{abstract}


\section{Introduction}\label{intro}
This paper  contributes partial solutions to two problems: How can one
decide if random variables are asymptotically independent and how can
their dependence structure be visualized, quantified and analyzed in practice? 
Our approach uses exploratory methods where grid based
  support estimates generate 
  hypotheses about dependence. These exploratory methods are combined with a testing scheme relying
  on asymptotic normality.

Both finance and insurance benefit
from having robust tools for understanding extremal
dependence; that is, the dependence of extreme values such as losses
  or claims. In finance, one of the central tasks of risk management is
to classify assets into classes that have minimal or no extremal
  dependence. 
 The joint behavior
of individual assets affects portfolio allocation strategies and
ultimately determines which equities are selected for a
portfolio. 
By selecting equities for the portfolio that are as independent as
possible, we reduce portfolio risk since the portfolio is unlikely to
experience many large losses at once. Our method can be viewed as a way to reduce the amount of systemic risk of a portfolio.
In insurance, large claims are a major risk factor because they can cause
insolvency. This risk is more serious if 
multiple lines of business can suffer
large claims at the same time. Furthermore, successful pricing of
insurance contracts as well as negotiations with reinsurance providers
depend on having an understanding of the worst case risks.  Thus,
 it is necessary to have an accurate
dependence model especially for large claims.  

We assume that all of the multivariate risks are {\em heavy-tailed}
which for this paper means the distributions of risk vectors are 
 multivariate regularly varying (MRV) \cite{MR2271424}; this is
 defined in Section \ref{subsec:mrv}.
In heavy tailed modeling, there are many complementary studies recommending methods for quantifying or
modeling multivariate extremal dependence.
Copula methods \cite{joe:li:2010,
  joe:1997,joe:2015, genestEtAl:2018,
  embrechts2001modelling,mcneil:frey:embrechts:2015,nelsen:2006}  have a large  literature 
  and 
are not restricted to heavy tails. Extreme value methods focus on quantifying  
asymptotic dependence between pairs using numerical summaries such as 
 the coefficient of tail dependence (\cite[p. 163]{coles:2001}, \cite[p258]{dehaan:ferreira:2006},
\cite{MR1707950,MR2156598, cooley:thibaud:2016})  
or similar concepts like the extremal dependence measure \cite{larsson:resnick:2012,MR2048253}
 or the extremogram\cite{MR3539306,davis:mikosch:2009,
   davis:mikosch:cribben:2012}. 
Other studies concentrate on estimating the limit measure of regular
variation or the angular measure \cite{MR2541452, MR1475660,
einmahl:dehaan:piterbarg:2001,dehaan:resnick:1993, MR2271424}
and additionally there is the hidden regular variation stream of inquiry
about whether multiple distinct heavy 
tail asymptotic regimes coexist
\cite{mitra:resnick:2011hrv,mitra:resnick:2013,MR3442427,
MR3077544,MR3737388,MR2271424,heffernan:resnick:2005,MR3271332,MR2002121}.
There are recent efforts to assess dependence by estimating the
support of the limit measure of regular variation \cite{MR3737388,
  Scholkopf:2001:ESH:1119748.1119749,MR3698112}
and growing interest in issues around dimension reduction of
high-dimensional heavy tailed vectors
\cite{MR3698112,MR3553240,cooley:thibaud:2016, MR3513601,HOFERT2017}.
 Beyond the MRV
setting, similar topics have been discussed from the extreme value
theoretical viewpoint; see e.g.  \cite[Section ~6]{MR1458613} and its
references. 

Our approach relies on an exploratory step in which we assess
dependence by estimating the support of the  limit measure.
The support of the limit measure can indicate if the risk vector
components are strongly asymptotically dependent or asymptotically independent.
This step is used to generate hypotheses about the dependence
structure which can then be tested more formally using test statistics
that are asymptotically normal. 
Support estimation is accomplished using what we call the {\it grid
  based estimator\/} which first bins the data and then counts
bin frequencies. We use this binning method to speed computation anticipating
cases where dimensions and sample sizes are large enough to cause
computing problems.

Assuming the data follows a standard multivariate regularly varying
distribution,
the data is first thresholded based on the magnitudes of sample
vectors and then divided into two parts. The first part is used to
establish the grid based estimation of the  support of the limit
measure and to generate hypotheses about the dependence structure. The
remaining data is used to test the validity of the support estimate
using an asymptotically normal test statistic. 

\subsection{Why conventional risk measures relying on correlation may
  mislead.}\label{subsec:mislead} 
In applications centered on extreme risk, conventional moment based
risk measures such as correlation are potentially misleading. This is
a persistent message in the extreme value and heavy tails literature.
Typically, the dependence structure of large observations determines
worst case risk and small observations may have minimal impact on
worst case risk even if highly dependent.  The following toy example
illustrates inadequacies of correlation.

\begin{example}\label{ex1} Let $\alpha>2$ and $l>1$. Suppose $X,Z$ and $B$ are independent random variables such that $\prob(X>x)=\prob(Z>x)=x^{-\alpha}$ for $x\geq 1$ and $1$ otherwise. Let $\prob(B=0)=1-\prob(B=1)=1/2$. Set 
$$Y_1:=X\mathbbm{1}(X\leq l)+lZB\mathbbm{1}(X>l)$$
and
$$Y_2:=Z\mathbbm{1}(X\leq l,Z\leq l)+X\mathbbm{1}(X>l).$$

Suppose the pairs $(X,Z)$ and $(X,Y_i)$, $i=1,2$, denote risks to a
company where the components of vectors correspond to different lines
of businesses. The aim of the company is to avoid insolvency from
  large losses, so the pair $(X,Y_1)$ should not be
considered more risky than $(X,Z)$ because the probability of two
simultaneous catastrophic losses exceeding $x>l$  for both vectors is
of order $x^{-2\alpha}$ . On the other hand,
the pair $(X,Y_2)$ is riskier than $(X,Y_1)$ or $(X,Z)$,
because $Y_2=X$ when $X>l$, resulting in the probability of two catastrophic losses
exceeding $x>l$ to be of order $x^{-\alpha}$.

However, one
 reaches contradictory conclusions if correlation is used to
quantify riskiness.  
Due to independence of $X$ and $Z$, $  \textnormal{Corr}(X,Z)=0$.
However, $\textnormal{Corr}(X,Y_1)\to 1,$ as $l\to \infty$. In
addition, the pair $(X,Y_2)$ is asymptotically fully dependent for all
$l>1$ using the terminology of \cite{MR3737388} since
$P[X>x,Y_2>x]\sim x^{-\alpha}$. Yet,
$\textnormal{Corr}(X,Y_2)\to 0,$ as $l \to \infty$. So, using
correlation as a measure of risk in this example leads to
overestimation of insignificant risk for $(X,Y_1)$ as well as underestimation of
potentially catastrophic risk for $(X,Y_2)$.  
\end{example}

Correlation, along with other popular risk metrics, fails to
adequately quantify risk in Example \ref{ex1} because
it has limited
capability of describing dependence of rare events. Similar phenomena
as in Example \ref{ex1} have been observed in
nature. In \cite{MR3553240}, the authors study meteorological data in
order to model extreme ground level ozone events. The study depicts
cases where the etremal observations have significantly different
dependence structure than small observations, see e.g. Figure 1 of
\cite{MR3553240}.  

In conclusion, modeling dependence structures with emphasis on
accuracy of tail behavior requires different tools than modeling
systems as a whole.  When the MRV or extreme value framework is
  applicable, the methods presented in Sections
\ref{supportestimationsection} and \ref{asindsection} {overcome} some of
the shortcomings of previous approaches and
allow practitioners to more fully understand how each risk
  contributes to overal risk management goals in finance and insurance.

\subsection{Structure of the paper}

The rest of Section \ref{intro}  defines notation, concepts and
definitions. In Section \ref{supportestimationsection}, the grid based
asymptotic support estimator for multivariate heavy-tailed data is
presented. Consistency and related properties are proved in Sections
\ref{supportestconsistency} and \ref{sec:consistencyforposq}. We review the definition of asymptotic independence
as well as connections with limiting behavior of {heavy
  tailed random vectors} 
in Section \ref{asindsection}
{and we introduce a  test for asymptotic independence based on} 
asymptotic normality in Section
\ref{asnormalsection}.  We illustrate
 the
techniques developed in Sections
\ref{supportestimationsection} and \ref{asindsection} 
by means of simulated and real examples
in Section \ref{realdatatestsection}.

\subsection{Basic definitions}

Suppose $(\Omega,\mathcal{B},\prob)$ is a probability space where all
the subsequent random variables are defined. Throughout the paper
random variables take values in a metric space
$(\mathbb{R}^{N},d)$, where $N\geq 2$ is the dimension of
the space, and
$d=d_{\mathbb{R}^N}$ is the $L_2$ or {\em Euclidean distance}:
for ${\bf x}
=(x^{(1)},x^{(2)},\ldots,x^{(N)})$ and 
${\bf y}=(y^{(1)},y^{(2)},\ldots,y^{(N)})$,
$d({\bf x},{\bf y})=\sqrt{\sum_{i=1}^{N}(x^{(i)}-y^{(i)})^2}$.
{The $L_2$-}distance is used in mappings that project sets into lower
dimensional spaces in a way that does not distort the image. However,
unless otherwise stated, $||\cdot||$ denotes the 
$L_1$-norm, where
$||{\bf x}||=\sum_{i=1}^{N}|x^{(i)}|. $
The $L_1$-norm is often natural
because in applications the total risk is typically the sum of
marginal risks. So, any condition on the size of the $L_1$ norm can be
directly viewed as a condition on the total risk.

Upper indices are
used to identify components of vectors. Lower indices are reserved for
order statistics. For $1\leq i \leq N$, the $i$th largest component of
${\bf x}$ is $x_{(i)}$. All inequalities and operations involving
vectors are understood componentwise as in Section 1.2 of
\cite{MR2002121}. It is convenient to set $\bone=(1,\dots,1)$
where the dimension is clear from context.  For a finite set $A$, the number of elements in $A$ is denoted by $\#A$.

For a metric space $\mathbb{E}$, we set $M_+(\mathbb{E})$ to be the
set of all non-negative Radon measures on $\mathbb{E}$; that is, measures that are
finite on compact subsets of $\mathbb{E}$.
The collections of open, closed and compact sets are denoted by
$\mathcal{G}$,$\mathcal{F}$ and $\mathcal{K}$, respectively. For
details about the {\it Hausdorff metric\/} on $\mathcal{K}$, see
\cite{matheron:1975, molchanov:2005}. For a set
$A\subset \mathbb{R}^N$ the whole space can be partitioned as
$\mathbb{R}^N=\textnormal{int}(A)\cup\textnormal{ext}(A)\cup\partial
A$, the  topological interior, exterior and boundary of the set $A$. The diameter of $A$ is denoted by
$\textnormal{diam}(A)$, the complement of $A$ by $A^c$ and the closure of $A$ by $\textnormal{cl}(A)$. The {Euclidean}
ball with center ${\bf x} \in \mathbb{R}^N$ and radius $\delta>0$ is
$B({\bf x},\delta)$. The notation $:=$ is used when the left hand
side is defined by the right hand side of the equation.

\subsection{Multivariate regular variation}\label{subsec:mrv}

The standard definition of multivariate regular variation is
defined in \cite[Theorem 6.1]{MR2271424}. We allow
 possibly negative
values of components. 

\begin{definition}\label{mrvdef}Suppose ${\bf
    Z}=(Z^{(1)},Z^{(2)},\ldots,Z^{(N)})$ is a random vector in
  $\mathbb{R}^N$. Set $\mathbb{E}:=[-\infty,\infty]^N\backslash \{{\bf
    0}\}$.  We say that ${\bf Z}$ is standard multivariate regularly
  varying with limit measure $\nu$ if there exists a function
  $b(t)\uparrow\infty$, as $t\to \infty$, such that  
\begin{equation}\label{mrvdefeq}
t \prob \left( \frac{{\bf Z}}{b(t)}\in
  \cdot\right)\stackrel{v}{\to}\nu (\cdot)
\end{equation}
in $M_+(\mathbb{E})$, where $\stackrel{v}{\to}$ stands for vague
convergence of measures. 
\end{definition}

Note that normalizing all components using the
same function $b$ implies that the components must be tail equivalent,
see \cite[Remark 6.1.]{MR2271424}.

Multivariate regular variation has an equivalent definition via the probability
measure $\s$, called the {\em angular measure} or  {\em spectral
  measure} defined on the $L_1$-unit sphere
\begin{equation}\label{cndef}
C^N:=\left\{{\bf z}\in \mathbb{R}^N: ||{\bf z}||=1 \right\}.
\end{equation}
Then equivalently, ${\bf Z}$ is standard multivariate regularly varying if there exist a function $b(t)\uparrow\infty$, as $t\to \infty$, such that for 
$$(R,\Theta):=\left(||{\bf Z}||,\frac{{\bf Z}}{||{\bf Z}||}\right)$$
we have 
\begin{equation}\label{prodmrvdef}
t \prob \left( \left(\frac{R}{b(t)},\Theta\right)\in \cdot\right)\stackrel{v}{\to}c \nu_\alpha \times \s 
\end{equation}
in $M_+((0,\infty]\times C^N )$, as $t\to \infty$, where $c>0$, $\s$ is a probability measure on $C^N$ and $\nu_\alpha((x,\infty])=x^{-\alpha}$. The number $\alpha>0$ is called the {\em tail index} of the multivariate regularly varying distribution.

In addition to the $N$-simplex in $L_1$, set
$$C^N_+:=C^N \cap \mathbb{R}^N_+=\left\{{\bf z} \in \mathbb{R}^N_+:z^{(1)}+z^{(2)}+\cdots+z^{(N)}=1\right\} $$
to denote the part of simplex $C^N$ where all coordinates are
non-negative.
The {\em face} of the simplex corresponding to indices  $A\subset
\{1,2,\ldots,N\}$ is
\begin{equation}\label{e:face}
C^N(A):=\left\{{\bf z} \in C^N : z^{(i)}=0, \textnormal{ when } i \notin A\right\}.
\end{equation}

\subsection{The support of a measure}\label{subsec:support}
The asymptotic dependence structure of ${\bf Z}$ is controlled by
  the angular measure $\s (\cdot)$ and considerable information about
  extremal dependence is contained in the support.

\begin{definition}[Support of measure in $\mathbb{R}^N$] 
If $\mu$ is a measure on $\mathbb{R}^N$,
the support $\suu(\mu)$ of $\mu$ is the set
\begin{equation*}
\suu(\mu):=\left\{{\bf z} \in \mathbb{R}^N : \mu(B({\bf z},\delta))>0 \textnormal{ for all }\delta>0\right\}.
\end{equation*} 
Also,  $\suu(\mu)$ is the smallest closed set 
carrying the mass of $\mu$, 
$$\suu(\mu)=\bigcap_{A \in \mathcal{F},\,  \mu(A^c)=0}  A.$$
\end{definition}
{The support of a measure on a suitable subset of $\mathbb{R}^N$ is defined similarly.} We will be interested in the support $\suu$ of $\s$ and the part
of the support of $\s$ on simplex $C_+^N$ is denoted by $\su$.  

\subsection{$N$-simplex and simplex mappings}\label{smapping}

Section \ref{supportestimationsection}  estimates the
support of the angular measure $\s$ by approximating  
the support of $\s$ on {the positive $N $-dimensional
$L_1$ sphere $C^N_+$. The set $C^N_+$ is first projected to an $N-1$ 
dimensional space $[0,1]^{N-1}$  using a bijective simplex mapping 
defined below. This enables us to partition the $N-1$ dimensional
image space into a }grid consisting of equally sized 
rectangles{. The preimages in $C_+^N$ of the rectangles in
  the $N-1$ 
  dimensional space partition the $N$ dimensional set $C^N_+$.
  So the use of a simplex mapping provides a simple way to
  partition the positive $L_1$ simplex and offers improved
  visualization. In particular, when $N=3$, the visual representation
  of the support estimate is a planar set.}  

Rectangles accepted into the estimated support 
are determined from the data by the
concentrations of probability mass of the empirical estimate of the
limiting measure $\s$. This contrasts with \cite{MR3737388} which
assumed the support was a connected interval and estimated this
interval using
the range of the thresholded data.
Advantages of this current approach are computational efficiency and that
finding the sets of highest concentration provides a way to eliminate
noise arising from unlikely observations that lie outside of the
asymptotic support.

\begin{definition}\label{Tdefinition}
Let $N\geq 2$. Suppose $T$ is a bijective mapping $T \colon C_+^N\to[0,1]^{N-1}$ with property
$$d_{\mathbb{R}^{N}}({\bf x},{\bf y})=a  d_{\mathbb{R}^{N-1}}(T({\bf x}),T({\bf y})) $$
for  all ${\bf x},{\bf y}\in C_+^N$ and some constant $a>0$. Such a
mapping $T$ is called a {\em{simplex mapping}} associated with $C_+^N$.  
\end{definition}

When $N=3$ a simplex map aids visualization because supports can be
visualized in $\mathbb{R}^2$.
There is flexibility when choosing the 
mapping $T$ so the grid
positioning can be adjusted with respect to observed data if
necessary. By shifting the grid, one can avoid concentration of points
on grid boundaries. 

\begin{example}\label{exprojections}
\begin{enumerate}[a)]
\item If $N=2$, one can set 
$$T\left(\begin{bmatrix}z_1\\z_2\\\end{bmatrix}\right)
=z_1.$$
Here, $a=\sqrt{2}$.
\item If $N=3$, setting 
$$ 
T\left(\begin{bmatrix}z_1\\z_2\\z_3\\\end{bmatrix}\right)
=\begin{bmatrix}\frac{1}{2}(z_2-z_1+1)\\ \frac{\sqrt{3}}{2} z_3\\ \end{bmatrix}
$$
gives a mapping $C_+^3\mapsto [0,1]^2$. The image $T(C_+^3)$ is a
region in $[0,1]^2$ inside an equilateral triangle with edges on
$(0,0),(1,0)$ and $(1/2,\sqrt{3}/{2})$. The isometry property in Definition \ref{Tdefinition} can be shown to hold by observing that $z_3=1-z_2-z_1$ and writing the expressions for the squared $L_2$ distances in $\mathbb{R}^3$ and $\mathbb{R}^2$. The property holds with $a=\sqrt{2}$. Mapping T has an inverse
$T^{-1} \colon T(C_+^3) \to C_+^3$ given by 
$$
T^{-1}\left(\begin{bmatrix}z_1\\z_2\\\end{bmatrix}\right)
=\begin{bmatrix} 1-z_1-\frac{z_2}{\sqrt{3}}\\z_1-\frac{z_2}{\sqrt{3}}\\ \frac{2z_2}{\sqrt{3}}\\ \end{bmatrix}.
$$ 
\end{enumerate}
\end{example}

The mapping of Example \ref{exprojections} is used to visualize
3-dimensional data in 2-dimensions in Section \ref{realdatatestsection}.

\section{Support estimation}\label{supportestimationsection}
This section defines  the grid based estimator and shows  asymptotic 
consistency under general assumptions.
Suppose  $N\geq 2$ is the dimension of the
  data and $m\geq 2$ is an integer that
  determines the {\em resolution} of the asymptotic
support estimate. We 
 map the $N$-dimensional simplex into
$[0,1]^{N-1}$ and then the  image  is partitioned into
$m^{N-1}$ smaller sets. The partition is called a \emph{grid} and the
sub-squares (or cubes in higher dimensions) are  \emph{cells}. Some  grid cells are
accepted as part of the support while the rest are rejected based on a
data driven 
rule described in Section \ref{supportestimationsection}.  

The topic of support estimation has antecedents though not usually
in the context of estimating the support of an {\it asymptotic
  distribution\/}.
See \cite{baCuevas:2001,meister:2006,
  chevalier1976estimation,
  baEtAl, KorEtal, gijbels:peng:2000,
  cuevas1997plug}. A support estimation problem in
\cite{hsingEtAl:1998} assumes a uniform distribution over a convex
set. Estimating the support from a sample by placing a small ball around each
sample point was suggested in \cite{devroye:wise:1980} and followed up
in \cite{biau:cadre:mason:pelletier:2009}; this method 
has some overlap with our proposal.
The method in \cite[Proposition 6.3]{MR3077544} for estimating an
asymptotic support omits a condition.

\subsection{Support estimator and related quantities}

\subsubsection{Reduction to positive quadrant.}
Random vectors in $\mathbb{R}^N$ can be considered to have $2N$ tails but
usually without loss of generality it is enough to write 
proofs for the positive quadrant since negative components can be
reflected into  positive values by multiplying by $-1$.
 Let ${\bf s}\in \{-1,1\}^N$ be a vector of plus or minus $1$'s and
 for ${\bf z}\in \mathbb{R}^N$ define
 $${\bf s} \cdot {\bf z}=(s^{(1)}z^{(1)},\dots ,s^{(N)}z^{(N)}).$$
If ${\bf Z}$ is a multivariate regularly varying random vector in
$\mathbb{R}^N$, extreme behavior of ${\bf Z}$ in a quadrant other than
$\mathbb{R}_+^N $ can be studied by reducing to the case of the
positive quadrant by multiplying by an appropriate
${\bf s}$. For simplicity, we present the theory, 
for the case where the entire support $\bf{supp}(\s)$ is
in $C^N_+$. The general case is readily reduced to this one. 

 For a simplex
 mapping $T$  and  a multivariate regularly varying random vector
 ${\bf Z} \in
 {\mathbb{R}^N_+}$, define the $N-1$ dimensional random variable ${\bf
   U}$ as 
{\begin{equation}\label{Udef}
{\bf U}:=T({{\bf Z}}/ {||{\bf Z}||} ).
\end{equation}

\subsubsection{
Partition  $[0,1]^{N-1}$ into cells.} 
 Given a vector ${\bf x} \in [0,1]^{N-1}$ and $m\geq 2$, 
 define a  cell $\Mxm\subset \mathbb{R}^{N-1}$ by 
\begin{align}\label{eq:cellDef}
\Mxm:=& {\bf x}+\left[0,1/m\right)^{N-1}=[{\bf x}, {\bf x} +\frac 1m
        \bone )\\
=&\Big[x^{(1)},x^{(1)}+\frac{1}{m}\Big)\times\cdots 
   \times\Big[x^{(N-1)},x^{(N-1)}+\frac{1}{m}\Big)\nonumber
\end{align}
which is just the  box $[0,1/m)^{N-1}$ shifted  by the vector ${\bf
  x}$.
For any natural number $n$, set $[n0]=\{0,\dots,n-1\}$ and define
\begin{align}
  G_m:=&([m0]/m)^{N-1}\label{e:Gm}\end{align}
So 
$$G_m=\left \{{\bf x} \in [0,1]^{N-1 }:x^{(i)}\in
\left\{0,\frac{1}{m},\frac{2}{m},\ldots, \frac{m-1}{m} \right\}, 
   i=1,\ldots,N-1\right\}.$$

   \subsubsection{Approximate the support.}   After
partitioning the set $T(C^{N}_+)$ by grid cells, we rasterize the
asymptotic support of $\s_T =\s \circ T^{-1} $ for computational efficiency and then
estimate this approximation to the asymptotic support.  The estimation is done
by mapping thresholded observations and creating cell counts.

The definitions of the rasterized support, the
grid based support estimator and the proof of estimator consistency
depend on
{Proposition 6.2 of \cite[p.~158]{MR3077544}  or
  \cite[p. 308]{MR2271424} which give
\begin{equation}\label{lemma1eq2}
\hat{\s}_n(\cdot):=\frac{1}{k}\sum_{i=1}^n \mathbbm{1}(||{\bf
  Z}_i||>||{\bf Z}_{(k+1)}||)\epsilon_{{\bf Z}_i/||{\bf
    Z}_i||}(\cdot)\Rightarrow  \mathbb{S}(\cdot) ,
\end{equation}
as $n\to\infty$, $k=k(n)\to\infty$, $n/k\to\infty$ in $\prob(C^N)$, the space of probability measures on $C^N$ and 
the  limit is non-random so convergence also holds in probability.  This
convergence is preserved under the mapping $T$ and we define 
\begin{equation}\label{e:S}
\mathbb{S}_T=\mathbb{S}\circ T^{-1},
\quad \
\hat{\s}_{n,T}=\hat {\s}_{n}\circ T^{-1}
=\frac{1}{k}\sum_{i=1}^n \mathbbm{1}(||{\bf
  Z}_i||>||{\bf Z}_{(k+1)}||)\epsilon_{{\bf U}_i}.
\end{equation}
Then we have
\begin{equation}\label{e:ST}
  \hat{\s}_{n,T} \Rightarrow  \mathbb{S}_T
\end{equation}
in $\prob\bigl(T(C^N)\bigr)$. With respect to $\prob_{k+1}(\cdot)=
\prob(\cdot \big | \Vert {\bf Z}_{(k+1)}\Vert )$, the points of
$\mathbb{S}_{n,T} $ are
$k$ iid random elements with common distribution
$\prob({\bf U}_1 \in \cdot \big | \Vert {\bf Z}_1\Vert >r)$ where $r$ is evaluated
at $\Vert {\bf Z} _{(k+1)}\Vert.$ We denote, then
\begin{equation}\label{e:extra}
  \hat{\mathbb{S}}_{n,T}=\frac 1k \sum_{i=1}^k \epsilon _{{\bf
      U}_i^{(k)}}.
  \end{equation}

\begin{definition}
In the positive quadrant,  $\smq $ is the
closure of
\begin{equation}\label{e:Tsupp}
  \bigcup\left\{\Mxm:{\bf x}\in G_m, \, \st(\Mxm)>q \right\}.  
\end{equation}
 So, $\smq$ is a set in $\mathbb{R}^{N-1}$. In the special case $q=0$, the set
defined by $\sm:={\textnormal{\bf{Tsupp}}(m,0)}$ is called the {\em
  rasterised support}  in $\mathbb{R}^{N-1}$ and is the smallest grid set with
resolution $m$ containing the support of $\s_T$. 
\end{definition}

\begin{definition}[Support estimator]\label{Adef} Let ${\bf Z}_1,{\bf
    Z}_2,\ldots,{\bf Z}_n$ be iid multivariate regularly varying
  vectors { in $\mathbb{R}^N_+$. Let ${\bf U}_1,{\bf
    U}_2,\ldots,{\bf U}_n$ be the corresponding random vectors in $\mathbb{R}^{N-1}$ obtained from transformation \eqref{Udef}.} Suppose $k$ and $m$ are natural numbers such that $k\geq 1$
  and $m\geq2$. For $q\in [0,1]$, the support
estimator 
$\smqest$ of $\smq$
is the {closure
of the set}
\begin{equation}\label{e:hatTsupp}
  \bigcup\Bigl\{\Mxm:{\bf x}\in G_m,
\hat{\s}_{n,T} (M({\bf x},m))>q\Bigr\}. 
\end{equation}
The estimator of $\sm $ is $\smest $ which is $\smqest $ with $q=0$.
\end{definition}

The support estimator $\smqest$ is a random {closed} set  based on a random
sample ${\bf Z}_1,{\bf Z}_2,\ldots,{\bf Z}_n$. It has three
parameters: $k$, $m$ and $q$. Parameter $k=k(n)$ is the number of
{extreme observations used in estimation.}
For the asymptotic analysis we
assume that $k=k(n)\to \infty,\;n/k(n)\to \infty$, as $n\to \infty$. Parameter $m$
denotes the resolution at which the estimate is formed. In asymptotic 
results, $m\to\infty$ so that  the resolution grows and the cell size decreases. 
The parameter $q$ serves as a rejection threshold. It determines how many
observations are needed in a single grid cell for the cell to be
accepted as part of the support estimate. In practice it helps to
reject unlikely observations and noise. If $p$ observations are
required in a given sample of $n$ one can set $q=p/k(n)$.  

Support estimators in Definition \ref{Adef} are decreasing in $q$. For
fixed $k,m$ and  $0\leq  q_1<q_2<1$,
\begin{equation}\label{e:ordered}
 {\widehat{\textnormal{\bf{Tsupp}}}(k,m,q_2)} \subset
 {\widehat{\textnormal{\bf{Tsupp}}}(k,m,q_1)}.
 \end{equation}

\subsection{Consistency of the grid based support
  estimator for $q=0$}\label{supportestconsistency}
The following results are derived for the case where the limiting angular
measure concentrates on the positive quadrant $C^N_+$. The general case
is not mathematically much different, but requires more
notation.

We begin by discussing continuity properties of the rasterization
procedure.

\subsubsection{The $\Rast $ operator.}\label{subsubsec:rast}
The $\Rast$ operator maps sets into rasterized versions. We define for
fixed resolution $m$,
$\Rast(\cdot,m): \mathcal{K}(T(C_+^N))\to \mathcal{K}([0,1]^{N-1})$ by
$$\Rast (K,m)=\textnormal{cl}\Bigl(\bigcup \{M({\bf x},m):
M({\bf x},m)\cap K\neq \emptyset\}
\Bigr).$$
So in the notation of \eqref{e:Tsupp},
$\sm=\Rast(\text{supp}(\st) ,m).$

We begin by discussing consistency results
with $m$ fixed.

\begin{proposition}\label{lem:cont}
  Suppose $K$ is a compact set satisfying
  \begin{equation}\label{e:condit}
\forall {\bf x} \in G_m: \quad  K\cap \textnormal{cl}(M({\bf x},m) ) \neq \emptyset \text{ implies }
  K\cap \textnormal{int}(M({\bf x},m) ) \neq \emptyset.
  \end{equation}
  Then $\Rast(\cdot,m)$ is continuous at $K$; that is,
  if $K_n \to K$ in the Hausdorff metric, then also $\Rast(K_n,m) \to
  \Rast(K,m).$
\end{proposition}

Condition \eqref{e:condit} says that if $K$ intersects 
a cell, it does not do so only on the boundary of the cell.
\begin{proof}
  We use the criterion \cite[page 6]{matheron:1975} that $K_n \to K$ in
  the Hausdorff metric
  iff
   \begin{enumerate}
    \item   ({\sc Condition 1.})
   $z\in K$ implies $\exists z_n \in K_n$ and $z_n\to
   z$. 
       \item   ({\sc Condition 2.})
    For a subsequence $\{n_j\}$, if $z_{n_j} \in K_{n_j}$ and
      $\{z_{n_j}\}$ converges, then $\lim_{n_j \to \infty} z_{n_j} \in
      K$.
    \end{enumerate}

    \underline{Condition 1:} Assume $K_n \to K$ and $\by \in
    \Rast(K,m)$. The easy case is where    $\by \in K$. Then there exist
$\by_n \in K_n\subset \Rast(K_n,m)$ such that $\by_n\to \by \in
K\subset \Rast(K,m)$. This verifies Condition 1 in the easy case.

Now for the more difficult case of Condition 1, assume $\by \in
\Rast(K,m)\setminus K$.
Then there exists $\bx_0 \in G_m$ such that $\by \in
M(\bx_0,m)$. Suppose temporarily $\by \in \textnormal{int}(M(\bx_0,m))$; this
restriction will be removed. Then for some $\delta>0$, (i)
$d(\by,K)\geq 2\delta$; (ii) $B(\by, \delta/36)\subset
\textnormal{int}(M(\bx_0,m))$ (since $\by$ is in the interior of the cell); and
therefore (iii) $B(\by,\delta/36)\cap K =\emptyset$ (from (i)). This
implies
\begin{equation}\label{e:claim}
  K_n \cap M(\bx_0,m)\neq \emptyset,\quad \text{ for all large $n$}.
  \end{equation}
The reason is that
$M(\bx_0,m)\subset \Rast(K,m) $ but $M(\bx_0,m) \cap K \neq \emptyset$
by the choice of $\bx_0$. Condition \eqref{e:condit} implies
$\textnormal{int}(M(\bx_0,m))\cap K \neq \emptyset $ so $\exists \by^* \in
\textnormal{int}(M(\bx_0,m))\cap K.$ Since $
\by^* \in K$ and $K_n \to K$, $\exists \by_n^* \in K_n$ such that
$\by_n^* \to \by^*$. Since $\by^*$ is in the interior of the cell and
because $\by_n^*$ is close to $\by^*$, for all large $n$, $\by_n^* \in
K_n\cap \textnormal{int}(M(\bx_0,m))$ which verifies \eqref{e:claim}.

Now $K_n$ is close to $K$ and from \eqref{e:claim} has 
points in the cell $M(\bx_0,m)$ which miss $B(\by, \delta/36)\subset M(\bx_0,m)$. Find
$\by_n \in B(\by,\delta/36)$ with $\by_n\to \by$. This means $\by_n
\in \Rast(K_n,m)\setminus K_n \subset \Rast(K_n,m)$ and $\by_n\to\by
\in \Rast(K,m)$ as required.

If $\by \in \partial M(\bx_0,m)$ approximate $\by $ by something in
the interior and proceed as above.

\underline{Condition 2:} Given $\by_{n_j} \in \Rast(K_{n_j},m)$ with
$\by_{n_j} \to \by_\infty$ and we must show $\by_\infty \in
\Rast(K,m).$ This means we must find $
\bx_0\in G_m$ such that $\by_\infty \in M(\bx_0,m)$ and
$M(\bx_0,m)\cap K \neq \emptyset.$ Because cells cover the space and
there are a finite number of cells, there is a cell hit by the
elements $\by_{n_j}$ infinitely often. Identify this cell as
$M(\bx_0,m)$. So for this cell and a further subsequence
$\{n_{j'}\}\subset \{n_j\}$, $\by_{n_{j'}} \in M(\bx_0,m)\subset \Rast(K_{n_{j'}},m)$ and therefore
$\by_\infty \in \textnormal{cl}(M(\bx_0,m)).$

To verify $M(\bx_0,m) \subset
\Rast(K,m) $ as required do the following:  Since $\by_{n_{j'}} \in
\Rast(K_{n_{j'}},m)$, there exists $\by^*_{n_{j'}} \in M(\bx_0,m)\cap
K_{n_{j'}} $ and by compactness a further subsequence converges
$\by^*_{n_{j''}} \to \by_\infty^*$ and $\by_\infty^* \in
\textnormal{cl}(M(\bx_0,m))\cap K$ since $ \by^*_{n_{j''}}  \in K_{n_{j''}} \to K.$
Using \eqref{e:condit} once more, this leads to existence of $\by^{**}
\in \textnormal{int}(M(\bx_0,m))\cap K$ which identifies $M(\bx_0,m) \subset
\Rast(K,m) $   so $\by_\infty \in \Rast(K,m).$
\end{proof} 

Now we explain one interpretation of how $\Rast(K,m)$ approximates $K$
and why the approximation gets better with bigger $m$.

\begin{proposition}\label{prop:minfty}
  Given a compact set $K$, as $m\to\infty$,
  $$\Rast(K,m) \to K$$
  in the Hausdorff metric.\end{proposition}

  \begin{proof}
Again we verify the two conditions given at the beginning of the last
proof which are equivalent to convergence in the Hausdorff metric.

\underline{Condition 1:} For $\by \in K$,  there exists $\by_m\equiv
\by \in \Rast(K,m)$ such that $\by_m\to \by$.

\underline{Condition 2:} Given $\{m_j\} $ such that $\by_{m_j} \in
\Rast(K,m_j)$ and $\by_{m_j} \to \by_\infty$;  we must show
$\by_\infty \in K.$ Observe for any $\bx \in G_m$,
$$d(\by_{m_j}, K) \leq \text{diam}\bigl(M(\bx,m_j)\bigr) \to 0, \quad
(m_j\to\infty), $$
and so $d(\by_\infty ,K)=0$ and  $\by_\infty \in K$ as required. \end{proof}

\subsubsection{Convergence of measures and convergence of their
  supports.}\label{subsub:convMeasSupp}
In view of \eqref{e:ST} and Propositions \ref{lem:cont} and
\ref{prop:minfty}, it is natural to think that we can proceed by
estimating the limit measure and then using the rasterized support of this
estimating measure as our estimated support of the limit measure.
To make this work requires a condition. Recall that for a
set $A$ and metric $d(x,y)$, the $\delta$-neighborhood of $A$ is 
$$A^\delta=\{x: d(x,A)<\delta\}.$$

\begin{lemma}\label{lem:measConvSupportConv}
  Suppose for $n\geq 0$ that $m_n(\cdot)$ are Radon measures on a complete separable
  metric space with the support of $m_n$ being the compact set
  $K_n$. If
  $m_n \to m_0$ vaguely, then for all $\delta>0$, there exist
  $n_0=n_0(\delta)$ such that for all $n\geq n_0$,
  $$K_0\subset K_n^\delta.$$
  Additionally, if for $\delta>0$ and sufficiently large $n$,
  \begin{equation}\label{e:cond}
    K_n \subset K_0^\delta,
  \end{equation}
then $K_n \to K_0$ in the Hausdorff topology.
\end{lemma}

\begin{proof}
If $x\in K_0$, there exists a $\delta$-neighborhood $B(x,\delta)$ of $x$
satisfying $m_0(\partial B(x,\delta)) = 0$ and $m_0(B(x,\delta))>0$. Then
$m_n(B(x,\delta)) \to m_0(B(x,\delta))>0$
and for large $n$, $m_n(B(x,\delta))>0$. Therefore there exists $x_n
\in K_n \cap B(x,\delta)) $ and $d(x_n,x)<\delta.$ So $x\in  K_n^\delta
$ and thus $x\in K_0$ implies $x\in K_n^\delta$ and $K_0\subset
K_n^\delta$.
This proves the first assertion and the claim $K_n\to K_0$ requires
the second containment in \eqref{e:cond}.
  \end{proof}

  \begin{remark}\label{rem:Oy}
    Without \eqref{e:cond}, it is not necessarily true
    that $K_n \to K_0$. Suppose the metric space is $[0,1]$ {and}
    $m_n=(1-\frac 1n)\epsilon_0 + \frac 1n \epsilon_1,\;
    m_0=\epsilon_0$
    so that $m_n \to m_0$. The supports fail to converge
    and \eqref{e:cond} is violated.
  \end{remark}

  \begin{corollary}\label{cor:rmVer}
    Suppose $M_n, n \geq 0$ are random measures on a metric space with
    metric $d(x,y)$ and with $M_0$ non-random
    and for $n\geq 0$, the support $K_n$ of $M_n$ is compact. Assume
    $M_n\Rightarrow M_0$.
    Then for
    the Hausdorff metric $D(\cdot,\cdot)$ we have $K_n \Rightarrow K_0$
    iff 
    \begin{equation}\label{e:condRnd}
\forall \delta >0, \quad  \prob\left(K_n\subset K_0^\delta   \right) \to 1,
  \quad n\to\infty.
\end{equation}
\end{corollary}    
\begin{proof}
  We must show for any $\eta>0$,
$\prob(D(K_n,K_0)\leq \eta) \to 1.$
  However, this probability convergence is equivalent to
$$\prob\left(K_n \subset K_0^\eta, K_0 \subset K_n^\eta\right)\to 1, \quad
\forall \eta>0.$$
It is only necessary to control $\prob(K_n\subset K_0^\eta).$
\end{proof}

\noindent {\bf Remark \ref{rem:Oy}} (continued). Let $X_{nj}, 1 \leq j \leq n$
be iid with distribution
$F_n=(1-\frac 1n) \epsilon_0 + \frac 1n \epsilon_1,$ for each $n\geq 1$.
Then
$M_n:=\frac 1n \sum_{j=1}^n \epsilon_{X_{nj}} \Rightarrow
\epsilon_0=:M_0,$ but
$K_n=\{X_{nj},1 \leq j\le n\} \not\to K_0=\{0\} $
and \eqref{e:condRnd} fails since
\begin{align*}
  \prob(\{X_{nj},1 \leq j\leq n\}&\subset K_0^\eta)=\prob\left(\bigvee_{j=1}^n
                               X_{nj} \leq \eta\right)\\
  =&(\prob(X_{n1}=0))^n =\left(1-\frac 1n\right)^n \to e^{-1} <1.\end{align*}

\begin{corollary}\label{cor:suff}
Assume the conditions of Corollary \ref{cor:rmVer} hold and $M_n$ is
of the form
$$M_n=\frac 1k  \sum_{i=1}^k \epsilon_{\bTheta_i^{(k))}}, $$
for $\{\bTheta_i^{(k)} ,1\leq i\leq k\}$ iid.  Then
$$K_n:=\suu(M_n)=\{\bTheta_i^{(k)},1\leq k\leq k\}  \Rightarrow \suu(M_0)=K_0$$
iff
\begin{equation}\label{e:anotherOne}
  k\prob\left(d(\bTheta_1^{(k)},K_0)>\eta\right)\to 0,\quad (\forall
  \eta>0,\,k\to\infty).
  \end{equation}
\end{corollary}

\begin{proof}
  Apply Corollary \ref{cor:rmVer} and note
  \begin{align*}
    \prob\Bigl( \{\bTheta_i^{(k)} & ,1\leq i\leq k\} \subset K_0^\eta \Bigr)=
                                                                    \prob\Bigl(
                                                                         \bigcap_{i=1}^k
                                                                         \{d(\bTheta_i^{(k)},K_0)\leq
                                                                         \eta\}
                                                                         \Bigr)\\
    =&   \Bigl(\prob\Bigl(d(\bTheta_1^{(k)},K_0)\leq \eta\Bigr)\Bigr)^k
    =\left(1-\frac{k\prob\left(d(\bTheta_1^{(k)},K_0)>\eta \right)  }{k}\right)^k.
  \end{align*}
This converges to $1$ iff \eqref{e:anotherOne} holds.
\end{proof}  

\subsubsection{Consistency.}\label{subsub:consistent}
We now consider consistency of the grid based support estimator.

\begin{theorem}\label{Hdistconvergence}
Let  $\{{\bf Z}_1,\dots, {\bf Z}_n \}$ be a random sample from a
regularly varying distribution  assumed for simplicity to 
concentrate on the positive quadrant  $\mathbb{R}_+^N$. Set
$\bTheta_i={\bf Z}_i/\|{\bf Z}_i\|$, and $\bU_i=T(\bTheta_i).$
Recall the definitions of $\s, \st, \s_{n,T}$ and
$\smest $. For $K=\text{{\bf supp}}(\st)$ assume \eqref{e:condit}
holds for every $m$ and that
the points of \eqref{e:extra} satisfy \eqref{e:anotherOne}.
Then for $n\to\infty, \,k=k(n)\to\infty, \,k/n \to 0$ and
$m=m(n) \to \infty$
\begin{equation}\label{eq:consistency}
  \Rast(\text{{\bf supp}}(\s_{n,T}), m(n))
  ={\widehat{\textnormal{\bf{Tsupp}}}(k(n),m(n))}\Rightarrow
  {\suu(\st)}
\end{equation}
in $\mathcal{K}(T(C_+^{N}))$ metrized by the Hausdorff metric $D$. Equivalently,
\begin{equation*} 
D\left({\widehat{\textnormal{\bf{Tsupp}}}(k(n),m(n))},{\suu(\st)}\right)\stackrel{P}{\to}0,
\quad n\to\infty
\end{equation*}
or 
synonomously, for any $\delta>0$, 
\begin{equation}\label{th2eq21}
\prob\left({\widehat{\textnormal{\bf{Tsupp}}}(k(n),m(n))}\subset({\suu(\st)})^\delta\right)\to 1, \quad n\to \infty
\end{equation}
and
\begin{equation}\label{th2eq22}
\prob\left({\suu(\st)} \subset
{\widehat{\textnormal{\bf{Tsupp}}}(k(n),m(n))}^\delta\right)\to
1, \quad n\to \infty,
\end{equation}
where recall for a set $A$, $A^\delta$ is the $\delta$-neighborhood of $A$.
\end{theorem}

\begin{proof}
We use a standard Slutsky style approach outlined for instance in
\cite[page 56]{MR2271424}: Suppose that
$\{X_{mn},X_m,Y_n,X; n \geq 1,m\geq 1\}$ are  random elements of a
metric space with metric $D(\cdot,\cdot)$ defined on a common
domain. Assume
\begin{enumerate}
\item  For each fixed $m$, as $n \to \infty$,
  \begin{equation}\label{e:i}
    X_{mn} \Rightarrow X_m.\end{equation}
\item As $m \to \infty $
  \begin{equation}\label{e:ii}
    X_m \Rightarrow X.\end{equation}
\item For all $\epsilon >0$,
\begin{equation}\label{e:iii}
\lim_{m \to \infty} \limsup_{n \to \infty}\prob\left(D(X_{mn},Y_n)>\epsilon\right)
=0.\end{equation}
Then, as $n \to \infty$,  we have
$$
Y_n \Rightarrow X.$$
\end{enumerate}
In our context, the metric space is compact subsets
$\mathcal{K}([0,1]^{N-1})$, $D$ is the Hausdorff metric and
\begin{align*}
  X_{mn}=&\Rast(\text{{\bf{supp}}}(\s_{Tn}),m)=\smest,& Y_n=&
                                                       \Rast(\text{{\bf{supp}}}(\s_{Tn}),m(n))  \\
  X_m=&\Rast(\text{{\bf{supp}}}(\s_T),m)& X=&\text{{\bf{supp}}}(\s_T).
\end{align*}
The assumptions give convergence results $\{{\bf U}_i^{(k)}, 1\leq i\leq k\}\Rightarrow
{\bf{supp}}(\s_T)$
and $\Rast({\bf{supp}} (\s_{T,n} ),m)\Rightarrow \Rast({\bf{supp}}(\s_T),m),$
which is convergence for fixed $m$ in \eqref{e:i}
and \eqref{e:ii} is
covered by Proposition \ref{prop:minfty} so we focus on proving
\eqref{e:iii}.

To do this, suppose $K=\{\bz_1,\dots, \bz_k\}$ is a
discrete set of distinct points and $m_1<m_2$. We {\sc claim} that 
\begin{equation}\label{e:claim2}
D\bigl(\Rast(K,m_1),\Rast(K,m_2)\bigr) \leq 1/m_1. \end{equation}

Start by assuming $k=N-1=1$ and $\bz_1= z\in (0,1)$; if $z=0$ or $z=1$, one can
check the result separately.  For $i=1,2$, the $m_i$-resolution cell covering $z $ is 
$[a_i,b_i)=\bigl[[zm_i]/m_i,( [zm_i]+1)/m_i\bigr)$ of width $1/m_i$
and the
usual  large
$n$-scenario is that $[a_2,b_2)\subset [a_1,b_1)$ so that the Hausdorf
distance $D\bigl([a_1,b_1),[a_2,b_2)\bigr)$
between the two intervals is bounded by
\begin{align*}
  (a_2-a_1)\vee (b_1-b_2)=&
   \left( \frac{[zm_2]}{m_2} -\frac{[zm_1]}{m_1}\right)
    \bigvee    \left(\frac{[zm_1]+1}{m_1}  -\frac{[zm_2]+1}{m_2}
    \right)\\
  \leq  &1/m_1,
\end{align*}
the width of the larger grid  interval. If the nesting between $[a_1,b_1)$
and $[a_2,b_2)$ is other than described, a similar argument shows the
Hausdorf distance is still bounded by $1/m_1$.

If $N-1=1$ and $k>1$ and the points are $z_1,\dots , z_k$,
suppose $m_1,m_2$ are large enough that
if a cell contains a point at either resolution, it contains only one
point.
The grid
intervals at resolution $m_i$ ($i=1,2$) are $[a_{il},b_{il})=
\bigl[[z_im_i]/m_i, ([z_im_i]+1)/m_i\bigr);\,
l=1,\dots,k$ and the large-$n$ scenario is that $[a_{2l},b_{2l})
\subset [a_{1l},b_{1l}), \,l=1,\dots, k.$ In this case 
the Hausdorff distance between the two grids is
$$\bigvee_{l=1}^k \left(\left(
\frac{[z_l m_2]}{m_2 } -\frac{[z_lm_1]}{m_1}\right) \bigvee
  \left(\frac{[z_l m_1]+1}{m_1} -\frac{[z_lm_2]+1}{m_2}\right) \right) \leq \frac
  {1}{m_1}$$
by the same reasoning as in the $k=1$ case. If the containments are not
as described in the large $n$-scenario, similar arguments suffice.

Now allow $N>2$ and $k>1$. Euclidean distance is equivalent to metric
$$d_\vee (\bx,\by)= \bigvee_{j=1}^{N-1} |x^{(j)}-y^{(j)}|$$
and using this metric in the Hausdorf metric shows that the Hausdorf
distance bound is still $\leq 1/m_1$. This verifies \eqref{e:claim2}.
 
To verify \eqref{e:iii}, for the probability, choose $m$ big enough
that $1/m <\epsilon$ and $n$ big enough that $m(n)>m$. Then \eqref{e:iii} is clear.
\end{proof}  

\begin{remark}
Convergence in the sense of Theorem \ref{Hdistconvergence} does not
guarantee that the approximation
{${\widehat{\textnormal{\bf{Tsupp}}}(k(n),m(n))}$} covers
the support $\suu(\st)$.
In fact, if $m(n)$ grows rapidly enough with respect to $k(n)$, as $n\to \infty$, the approximation may have zero Lebesgue measure in the limit. 
\end{remark}

\subsection{Consistency of the grid based estimator for $q>0$.}\label{sec:consistencyforposq}
\begin{proposition}[Fixed $m$-consistency of the grid
  estimator]\label{th1} {Suppose the vectors $\{{\bf Z}_1,\dots, {\bf Z}_n \}$ form a random sample from a
regularly varying distribution that 
concentrates on the positive quadrant $\mathbb{R}^N_+$. Definition \ref{Adef}
defines the random set $\smqest$.  } 
For 
  $m\geq 2$, fix  ${\bf x}  \in G_m$ and assume
\begin{equation}\label{exset} 
q\in (0,1),    \quad q \neq     \st(M({\bf x},m)).
\end{equation}
Set
\begin{equation}\label{e:1}
  \hat{\mathbbm{1}}_x=\mathbbm{1}(M({\bf x},m)\subset  \smqest),
  \quad
{\mathbbm{1}}_x=\mathbbm{1}(M({\bf x},m)\subset \smq) .
  \end{equation}
Then 
as $n\to\infty$, $k(n)\to \infty$
  and  $n/k(n)\to \infty,$
\begin{align}\label{sprobability}
  \prob (\mathbbm{1}_x=\hat{\mathbbm{1}}_x)
      \to 1. 
\end{align}
Since the set $G_m$ is finite, it is also true that
\begin{align}
  \prob \left(\bigcap_{x\in G_m}\{\mathbbm{1}_x=\hat{\mathbbm{1}}_x\}\right)
                                     \to 1. \label{e:all=}
\end{align}
\end{proposition}
\begin{proof}
  Suppose first that $\st(M({\bf x},m))>q$ so that
  $$\mathbbm{1}_x=\mathbbm{1}(M({\bf x},m)\subset \smq)=1.$$
  Then the probability in  \eqref{sprobability} can be written as
$$\prob (\hat{ \s}_{n,T}(M({\bf x},m))>q)$$
and by \eqref{e:ST},
$\hat \s_{n,T} (M({\bf 
  x},m)) \stackrel{P}{\to} \st (M({\bf x},m))$.
Since
$\st(M({\bf x},m))> q$ by assumption, \eqref{sprobability} follows.

If $\st(M({\bf x},m))<q$, then  $\mathbbm{1}(M({\bf x},m)\subset
\smq)=0$ and the proof mimics the first case, but the
probability is
$\prob ( \hat \s_{n,T} (M({\bf x},m)) \leq q).$
\end{proof}

 Condition \eqref{exset} {ensures that convergence occurs for fixed $m$.}

\begin{corollary}
Suppose the assumptions of Proposition \ref{th1} hold. Then
 for fixed $m$ and $q\in (0,1)\backslash \{\st(M({\bf x},m)) : \, {\bf x} \in G_m\}$,
\begin{equation}\label{cor1eq21}
\lim_{n\to\infty}\prob\left(\smqest=\smq \right) =1.
\end{equation}
\end{corollary}
\begin{proof}
The event in \eqref{cor1eq21}
is a finite intersection of events of the form {$\{
  \mathbbm{1}(G\subset \smqest)=\mathbbm{1}(G\subset \smq)\}$ that
  have probability $1$ in the limit $n\to \infty$. The result follows
  using Proposition \ref{th1}.} 
\end{proof}

Theorem \ref{Hdistconvergence} considers only the case where the grid size tends to zero and $q=0$. However, since the estimators are used with positive parameter values of $q$, the content of the theorem should also hold when $q$ is not zero, but a function of $n$ that tends to zero, as $n$ grows. This result follows immediately once we note using Definition \eqref{e:hatTsupp} that 
\begin{equation}\label{estequality}
\widehat{\textnormal{\bf{Tsupp}}}(k,m,0)=\widehat{\textnormal{\bf{Tsupp}}}(k,m,q)
\end{equation}
holds almost surely when $q$ is small enough with respect to $k$. More precisely, if $q=q(n)$ satisfies 
$$\limsup_{n\to \infty} k(n)q(n)<1,$$
then \eqref{estequality} holds eventually in $n$ and we can replace $q=0$ by $q=q(n)$ in the statement of the theorem.

\section{Asymptotic independence}\label{asindsection}

Asymptotic independence is a more general property than
independence and is suitable for considering the influence of extreme
values.  If a random vector has asymptotically independent components,
a large component of the vector gives little information about the
likelihood of
other components being large.

Asymptotic independence is a
dependence structure in which vector realizations
containing multiple large components are unlikely
and from a practical {risk}  viewpoint,
asymptotically independent components are as harmless as independent
components. Thus omitting asymptotically independent {subsets of}
  components {from the vector}
analysis is a way to reduce the dimension of a studied
system. Doing so should increase
 the accuracy of estimates of the asymptotic support of the angular
 measure which is useful because typically only a limited
 amount of data is available. The topic of dimension reduction in
models with extremal dependence is also discussed in
\cite{MR3698112,Scholkopf:2001:ESH:1119748.1119749}.

{We review the}  definition of asymptotic independence 
which is compatible with existing
literature (e.g. \cite[p~195]{MR2271424}) and is applicable to 
several groups of
components. Overlapping approaches {include} 
\cite{davis:mikosch:2009,MR3539306,davis:mikosch:cribben:2012,MR3698112}.
Definition \ref{asinddef} {assumes} marginals
are heavy-tailed. The behavior of vectors composed of
sufficiently light-tailed iid components is different. See
\cite{MR3412770} for the two dimensional
case.  

\subsection{Definition of asymptotic independence of MRV}

\begin{definition}\label{asinddef}[Asymptotic independence for MRV]
  Suppose ${\bf Z} \geq {\bf 0}$ has a regularly varying
  multivariate distribution with scaling function $b(\cdot)$. Let $A_1,A_2
  \subset \{1,2,\ldots,N\}$ and suppose $\#A_1=N_1$ and
  $\#A_2=N_2$, {where $N_1,N_2\geq 1$}. {The} component ${\bf Z}_{A_1}:=(Z^{(i)})_{i\in
    A_1}$ is {\em asymptotically independent} of  component ${\bf
    Z}_{A_2}:=(Z^{(i)})_{i\in A_2} $ if 
\begin{equation}\label{asinddefeq}
\lim_{t\to\infty}
t \prob \left(\frac{{\bf Z}_{A_1}}{b(t)}\in B_1,\frac{{\bf Z}_{A_2}}{b(t)}\in B_2  \right)=0,
\end{equation}
 for all {Borel} $B_1\subset \mathbb{R}^{N_1}$,  $B_2\subset
 \mathbb{R}^{N_2}$ such that $d(B_i,{\bf 0})>0,\;i=1,2$.
\end{definition}

\begin{remark} It may be assumed without loss of generality that the
  sets $B_1$ and $B_2$ in \eqref{asinddefeq} are $N_1$ and $N_2$
  dimensional rectangles. The statement is made more precise in Part
  \ref{cc1} of Theorem \ref{charthm}. The condition $d(B_i,{\bf 0})>0$
  means $B_i$ is remote from ${\bf 0}$ and ${\bf Z}_{A_i} $ is an
  extreme vector.
\end{remark}

Next, we define projections and methods that can be used to combine
multiple components of random vectors into a single group. It enables
the study of two groups in a simple setting even though the original
data set is high dimensional. Recall the definition of $C^N(A)$ from \eqref{e:face}.

\begin{definition}\label{midpointdef}
Let $A_1,A_2 \subset \{1,2,\ldots,N\}$, $A_1\cap A_2 =\emptyset$ and suppose $\#A_1=N_1$ and $\#A_2=N_2$, where $N_1,N_2\geq 1$ and $N_1+N_2=N$. Define vectors ${ \bf a}_1,{ \bf a}_2\in C_+^N$  by formulas
\[{ \bf a}_1^{(i)}= \begin{cases} 
      1/N_1, & i\in A_1 \\
      0, & i\notin A_1
   \end{cases}
\]
and
\[{ \bf a}_2^{(i)}= \begin{cases} 
      1/N_2, & i\in A_2 \\
      0, & i\notin A_2.
   \end{cases}
\]
Vectors  ${ \bf a}_1$ and ${ \bf a}_2$ are called the \emph{midpoints of faces} $C^N(A_1)$ and $C^N(A_2)$, respectively.
\end{definition}

Midpoints ${ \bf a}_1$ and ${ \bf a}_2$  are linearly independent
{column} vectors in $\mathbb{R}^N$ and the subspace $W_{{ \bf
    a}_1,{ \bf a}_2}:=\textnormal{span}({ \bf a}_1,{ \bf a}_2)$
spanned by the midpoints is a plane. Thus we define orthogonal
projections onto the subspace $W_{{ \bf a}_1,{ \bf a}_2}$ via the
 projection matrix $\pmb{Q}_{{ \bf a}_1,{ \bf
     a}_2}:=\pmb{M}(\pmb{M}^T\pmb{M})^{-1}\pmb{M}^T,$ where $\pmb{M}$
 is the $N \times 2$ matrix $\pmb{M}=[{ \bf a}_1,{ \bf a}_2]$. When
 the subspace is spanned by midpoints {of faces}, the projection
 matrix $\pmb{Q}_{{ \bf a}_1,{ \bf a}_2}$  has a simple form. By a
 direct calculation,
\begin{equation}\label{qmatrix}
\pmb{Q}_{{ \bf a}_1,{ \bf a}_2}=[{ \bf c}_1,{ \bf c}_2,\ldots,{ \bf c}_N],
\end{equation}
where 
\[{ \bf c}_i= \begin{cases} 
      { \bf a}_1, & { \bf a}_1^{(i)}\neq 0 \\
      { \bf a}_2, & { \bf a}_2^{(i)}\neq 0.
   \end{cases}
\]

\begin{example} Suppose $N=5$, $A_1=\{1,2,4\}$ and $A_2=\{3,5\}$. Now ${ \bf a}_1=[1/3,1/3,0,1/3,0]^T$, ${ \bf a}_2=[0,0,1/2,0,1/2]^T$ and 
$$\pmb{Q}_{{ \bf a}_1,{ \bf a}_2}=
\begin{bmatrix}
    \frac{1}{3} & 0 & \frac{1}{3} & \frac{1}{3} & 0 \\
    0 & \frac{1}{2} & 0 & 0 & \frac{1}{2} \\
    \frac{1}{3} & 0 & \frac{1}{3} & \frac{1}{3} & 0 \\
    \frac{1}{3} & 0 & \frac{1}{3} & \frac{1}{3} & 0 \\
    0 & \frac{1}{2} & 0 & 0 & \frac{1}{2} \\
\end{bmatrix}. $$
\end{example}

An  orthogonally projected point is connected to linear combinations
of midpoints ${ \bf a}_1$ and ${ \bf a}_2$ and such a point
${\bf x}\in\mathbb{R}_+$
has representation 
\begin{equation}\label{oprojdef}
\pmb{Q}_{{ \bf a}_1,{ \bf a}_2}{ \bf x}=\left(\sum_{i\in A_1} x^{(i)}\right){ \bf a}_1+\left(\sum_{i\in A_2} x^{(i)}\right){ \bf a}_2. 
\end{equation}

Next, we will define projections that allow projection of multidimensional data onto a line. The projected points can be used to inspect validity of asymptotic independence.

\begin{definition}\label{rnto01lemma}
Let $A_1$ and $A_2$ be as in Definition \ref{midpointdef} and $\pmb{Q}_{{ \bf a}_1,{ \bf a}_2}$ as in \eqref{qmatrix}. 

Mappings $h_1 \colon \mathbb{R}^N_+\backslash \{{\bf 0}\} \mapsto C^N_+$,
$h_2 \colon \mathbb{R}^N_+ \mapsto \mathbb{R}^N_+$ and $h_3 \colon
\{(1-t){ \bf a}_1+t{ \bf a}_2: t\in [0,1]\}\mapsto [0,1]$  are defined as 
$$h_1({ \bf x}):=\frac{ { \bf x}}{||{ \bf x}||}, \quad
h_2({ \bf x}):=\pmb{Q}_{{ \bf a}_1,{ \bf a}_2}{ \bf x}, \quad
h_3({ \bf x}):=h_4^{-1}({ \bf x}),$$
where $h_4$ is the linear interpolation $h_4(t)=(1-t){ \bf a}_1+t{ \bf
  a}_2$, $t\in [0,1]$. We define projection $\pa\colon
\mathbb{R}_+^N\backslash \{{\bf 0}\}\mapsto [0,1]$
by \begin{equation}\label{padef} 
\pa({\bf x}):=h_3(h_2(h_1({\bf x}))).
\end{equation}
\end{definition}

Function $\pa({\bf x})$ projects points of $\mathbb{R}_+\backslash
\{{\bf 0}\}$ first onto the $L_1$-simplex and then orthogonally onto the
line connecting midpoints ${ \bf a}_1$ and ${ \bf a}_2$. The order of
projections $h_1$ and $h_2$ can be switched.  

\begin{lemma}\label{h1h2switchlemma} Suppose ${\bf
    x}\in\mathbb{R}_+\backslash \{{\bf 0}\}$. Let $A_1$ $A_2$, $h_1$
  and $h_2$ be as in Definition \ref{rnto01lemma}.  
Then 
\begin{equation}\label{h1h2switcheq}
h_2(h_1({\bf x}))=h_1(h_2({\bf x})).
\end{equation}
\end{lemma}
\begin{proof} We note first that $\pmb{Q}_{{ \bf a}_1,{ \bf a}_2}{\bf
    x}\in \mathbb{R}_+\backslash \{{\bf 0}\}$ so that the function
  $h_1(h_2({\bf x}))$ is well defined. Also \eqref{oprojdef} and 
  $\pmb{Q}_{{ \bf a}_1,{ \bf a}_2}=\pmb{Q}_{{ \bf a}_1,{ \bf a}_2}^T$
imply
\begin{equation}\label{qnormeq}
||\pmb{Q}_{{ \bf a}_1,{ \bf a}_2}{\bf x}||=\sum_{i=1}^{N_1} { \bf a}_1\cdot { \bf x}+\sum_{i=1}^{N_2} { \bf a}_2\cdot { \bf x}=\sum_{i=1}^{N} x^{(i)}=||{\bf x}||.
\end{equation}
Now, using linearity of $h_2$ and Equation \eqref{qnormeq} we get 
$$h_2(h_1({\bf x}))=\pmb{Q}_{{ \bf a}_1,{ \bf a}_2}\frac{ { \bf x}}{||{ \bf x}||} =\frac{ \pmb{Q}_{{ \bf a}_1,{ \bf a}_2}{ \bf x}}{||{ \bf x}||} = \frac{ \pmb{Q}_{{ \bf a}_1,{ \bf a}_2}{ \bf x}}{||\pmb{Q}_{{ \bf a}_1,{ \bf a}_2}{\bf x}||}= h_1(h_2({\bf x})). $$
\end{proof}

Lemma \ref{h1h2switchlemma} states that the mapping $\pa$ of
Definition \ref{rnto01lemma} can be viewed in two different ways. This
{observation is} relevant for the proof of {Theorem \ref{charthm}
  below}. 

\subsection{Connection between asymptotic independence and the limit measure}

\begin{theorem}\label{charthm}
Suppose ${\bf Z}\geq {\bf 0}$ is a multivariate regularly varying random vector. Let ${\bf Z}$, ${\bf Z}_{A_1}$ and ${\bf Z}_{A_2}$ be as in Definition \ref{asinddef} and $A_1\cap A_2 =\emptyset$.  
The following are equivalent with \eqref{asinddefeq}:
\begin{enumerate}[$1)$]
\item\label{cc1} Suppose $B_1\subset \mathbb{R}^{N}$ and $B_2\subset
  \mathbb{R}^{N}$ are Borel sets bounded away from ${\bf 0}$ with
    the structure
$$B_1=B_1^{(1)}\times B_1^{(2)}\times \cdots \times B_1^{(N)},
\textnormal{ where } B_1^{(i)}=\mathbb{R} \textnormal{ for all }i\in
A_2 $$ 
and
$$B_2=B_2^{(1)}\times B_2^{(2)}\times \cdots \times B_2^{(N)}, \textnormal{ where } B_2^{(i)}=\mathbb{R} \textnormal{ for all }i\in A_1. $$

Then
\begin{equation*}
t \prob \left(\frac{{\bf Z}}{b(t)}\in B_1\cap B_2  \right)\to 0, \quad t\to \infty.
\end{equation*}
\item\label{cc2} Suppose $i\in A_1$, $j\in A_2$ and $c>0$.

Then
\begin{equation}\label{cc2eq2}
t \prob \left(\frac{Z^{(i)}}{b(t)}>c,\frac{Z^{(j)}}{b(t)}>c  \right)\to 0, \quad t\to \infty.
\end{equation}
\item\label{cc3}  The angular measure $\s$  concentrates
  on faces corresponding to $A_1$ and $A_2$,   
\begin{equation}\label{cc3eq3}
\s(C^N(A_1))+\s(C^N(A_2))=1.
\end{equation}
\end{enumerate}
\end{theorem}
\begin{proof} \eqref{asinddefeq} $\Leftrightarrow$ \ref{cc1}: Suppose sets $B_1 \subset \mathbb{R}_+^{N_1}$ and $B_2 \subset \mathbb{R}_+^{N_2}$ are bounded away from ${\bf 0}$. Define sets $D_{k,c}\subset\mathbb{R}_+^{N}$, where $k=1,2,\ldots,N$ and $c>0$ by 
\begin{equation}\label{cc1proofeq1}
D_{k,c}=D_{k,c}^{(1)}\times D_{k,c}^{(2)}\times \ldots \times D_{k,c}^{(N)},
\end{equation}
where 
\[D_{k,c}^{(i)}:= \begin{cases} 
      [c,\infty), & i=k \\
      \mathbb{R}_+, & i\neq k.
   \end{cases}
\]
Since the sets $B_1$ and $B_2$ are bounded away from ${\bf 0}$, there must be numbers $c_1>0$ and $c_2>0$ so that 
\begin{eqnarray}
t \prob\left(\frac{{\bf Z}_{A_1}}{b(t)}\in B_1,\frac{{\bf Z}_{A_2}}{b(t)}\in B_2  \right)&\leq& t \prob \left(\frac{{\bf Z}}{b(t)}\in \left(\cup_{k\in A_1} D_{k,c_1}\right)\cap \left(\cup_{k\in A_2} D_{k,c_2}\right) \right) \nonumber \\ 
&\leq& \sum_{k_1=1}^{N_1} \sum_{k_2=1}^{N_2} t \prob \left(\frac{{\bf Z}}{b(t)}\in D_{k_1,c_1}\cap D_{k_2,c_2}  \right).\label{cc1proofeq2}
\end{eqnarray} 
Each term on the right hand side of \eqref{cc1proofeq2} converges to $0$, as $t\to \infty$ by Condition \ref{cc1}. This shows  \ref{cc1} $\Rightarrow$ \eqref{asinddefeq}. The remaining direction is clear because product sets are special cases of sets in \eqref{asinddefeq}.

\ref{cc1} $\Leftrightarrow$ \ref{cc2}: Suppose \ref{cc2} holds and let $B_1$ and $B_2$ be as in Condition \ref{cc1}. Since $B_1$ and $B_2$ are bounded away from ${\bf 0}$ there must be indices $k_1\in A_1$, $k_2\in A_2$ and a number $c>0$ such that $B_1\subset D_{k_1,c}$ and $B_2\subset D_{k_2,c}$, where the sets $D_{k_1,c}$ and $D_{k_2,c}$ are defined as in \eqref{cc1proofeq1}. Then 
$$t \prob\left(\frac{{\bf Z}}{b(t)}\in B_1\cap B_2  \right)\leq t \prob\left(\frac{{\bf Z}}{b(t)}\in D_{k_1,c}\cap D_{k_2,c}  \right), $$
where the right hand side converges to $0$, as $t\to \infty$ by Condition \ref{cc2}. The other direction is clear because the sets in \ref{cc2} are special cases of sets in \ref{cc1}.

\ref{cc3} $\Rightarrow$ \ref{cc2}: Suppose first that Condition \ref{cc2} does not hold. Then there exist indices $k_1\in A_1$, $k_2\in A_2$ and $c>0$ such that \eqref{cc2eq2} does not hold, i.e. the limit does not exist or the limit exists but is not $0$. Even if the set in \eqref{cc2eq2} is a not a continuity set of the limit measure $\nu$, we may choose a smaller number $c'\in(0,c)$ so that the right hand side of  
$$ \{Z^{(k_1)}>c b(t),Z^{(k_2)}>c b(t) \}\subset  \{Z^{(k_1)}>c' b(t),Z^{(k_2)}>c' b(t) \}$$
is a continuity set. So, when $c$ is replaced by $c'$ in \eqref{cc2eq2} the limit given by limit measure $\nu$ exists, as $t \to \infty$. Since the limit is not $0$ by assumption, it must be positive. So, $\nu(D_{k_1,c}\cap D_{k_2,c} )>0$, where the sets $D_{k_1,c}$ and $D_{k_2,c}$ are as in \eqref{cc1proofeq1}. Because the set $D_{k_1,c}\cap D_{k_2,c}$ gets positive value under measure $\nu$, the image under $h_1$ of this set must have positive angular measure, where $h_1$ is as in Definition \ref{rnto01lemma}. Specifically, 
\begin{equation}\label{cc2proofeq1}
\s(h_1(D_{k_1,c}\cap D_{k_2,c}))>0.
\end{equation} 
If $x\in h_1(D_{k_1,c}\cap D_{k_2,c})$, then $\sum_{i\in A_1} x^{(i)}>0$ and $\sum_{i\in A_2}x^{(i)}>0$. So, 
\begin{equation}\label{cc2proofeq2}
h_1(D_{k_1,c}\cap D_{k_2,c})\cap C^N(A_1)=\emptyset
\end{equation}
and
\begin{equation}\label{cc2proofeq3}
h_1(D_{k_1,c}\cap D_{k_2,c})\cap C^N(A_2)=\emptyset.
\end{equation}
Since $\s$ is a probability measure and some of the probability mass is concentrated outside of the faces $C^N(A_1)$ and $C^N(A_2)$ by \eqref{cc2proofeq1}, \eqref{cc2proofeq2} and\eqref{cc2proofeq3}, we have that
$$\s(C^N(A_1))+\s(C^N(A_2))<1. $$
So, Condition \ref{cc3} does not hold.

\eqref{asinddefeq} $\Rightarrow$ \ref{cc3}: Suppose Condition \ref{cc3} does not hold. Then there exist a set $B\subset C_+^N$ such that $\s(B)>0$,
$$B\cap C^N(A_1)=\emptyset $$
and
$$B\cap C^N(A_2)=\emptyset. $$
Since $B$ does not intersect either face, there are numbers $c_1,c_2\in(0,1)$ so that the set 
\begin{equation*}
B_{c_1,c_2}:=\left\{{\bf x}\in B: \sum_{i\in A_1} x^{(i)}>c_1, \sum_{i\in A_2}x^{(i)}>c_2\right\}
\end{equation*}
has positive angular measure, that is $\s(B_{c_1,c_2})>0$. Define $D\subset \mathbb{R}^N$ using $B_{c_1,c_2}$ by $D:=\{c{\bf x}:  c\geq 1, {\bf x}\in B_{c_1,c_2}\}$. Now $\nu(D)>0$. Furthermore, 
\begin{equation}\label{cc3proofeq1}
D\subset \left\{{\bf x}\in \mathbb{R}^N: \sum_{i\in A_1} x^{(i)}>c_1, \sum_{i\in A_2}x^{(i)}>c_2\right\}.
\end{equation}
So, when $B_1=\left\{{\bf x}\in \mathbb{R}^{N_1}: \sum_{i} x^{(i)}>c_1 \right\}$ and $B_2=\left\{{\bf x}\in \mathbb{R}^{N_2}: \sum_{i} x^{(i)}>c_2 \right\}$ in \eqref{asinddefeq} we have that
$$t \prob \left(\frac{{\bf Z}_{A_1}}{b(t)}\in B_1,\frac{{\bf Z}_{A_2}}{b(t)}\in B_2  \right)\geq t \prob \left(\frac{{\bf Z}}{b(t)}\in D\right),$$
where the right hand side does not converge to $0$, but to $\nu(D)>0$. This shows that \eqref{asinddefeq} does not hold.
\end{proof}

\begin{remark} Part \ref{cc1} of Theorem \ref{charthm} admits sets
  that have zeros as some of their components. For example, if $N=3$,
  $A_1=\{1,3\}$ and $A_2=\{2\}$, then $B_1$ can be $\{0\}\times
  \mathbb{R} \times [1,\infty)$. For this reason the condition $t \prob
  \left({\bf Z}/b(t)\in [c,\infty)^N\right)\to 0,$ as $t\to \infty$
  for all $c>0$ is {necessary, but not sufficient for} asymptotic independence.
\end{remark}

The following result helps reduce multidimensional dependence
structures to the two dimensional setting by considering sums of components. 

\begin{proposition}\label{prop:reduce}
Suppose ${\bf Z}=(Z^{(1)},Z^{(2)},\ldots,Z^{(N)})$ is a non-negative
MRV random vector and $N\geq 2$. Let $A_1,A_2 \subset
\{1,2,\ldots,N\}$, $A_1\cap A_2 =\emptyset$ and suppose $\#A_1=N_1$
and $\#A_2=N_2$, where $N_1,N_2\geq 1$ and $N_1+N_2=N$. 

Then the non-negative two dimensional random vector 
$$(Y_1,Y_2):=\left(\sum_{i\in A_1} Z^{(i)},\sum_{i\in A_2} Z^{(i)} \right) $$
is also MRV. Furthermore, ${\bf Z}_{A_1}$ and  ${\bf Z}_{A_2}$ are
asymptotically independent if and only if 
 $Y_1$ and $Y_2$ are
asymptotically independent. 
\end{proposition}

\begin{proof} The regular variation of $(Y_1,Y_2)$ follows from
\cite[Proposition 5.5, p.~142]{MR2271424}.  For the second claim,
observe first using \eqref{oprojdef} that if ${\bf x}\in C^N_+$, then
for $j=1,2$, $\pmb{Q}_{{ \bf a}_1,{ \bf a}_2}({\bf x})={ \bf a}_j$ if
and only if ${\bf x}\in C(A_j)$. So, it follows that
$\s(C(A_1))+\s(C(A_1))=1$ if and only if $\s_Y({(0,1)^T})+\s_Y({(1,0)^T})=1$,
where $\s_Y$ denotes the angular measure of $(Y_1,Y_2)$. Using Part
\ref{cc3} of Theorem \ref{charthm} completes the proof. 
\end{proof}

\subsection{Asymptotic normality of the validation
  statistic}\label{asnormalsection} 
In light of Theorem \ref{charthm} and Proposition
\ref{prop:reduce},  in this section we concentrate on the case $N=2$.}

We start with an auxiliary function $g(\cdot)$ in Definition
\ref{gdef} used to create a test statistic in  Theorem
\ref{astheorem}. 
The function $g$  fixes a
subset on 
$C^2$ and is the basis of a test for whether
 the asymptotic support of $\s$ is
included in the fixed set. Different choices for $g$ yield tests for
different dependence structures which could  include asymptotic
independence introduced in Section \ref{asindsection}.
Commonly encountered
$g$'s are illustrated in Figure \ref{fig:gexample}.  

\begin{definition}\label{gdef} Suppose $[a_1,b_1],[a_2,b_2],\ldots,[a_m,b_m]$ are separate subintervals of $[0,1]$, where $m\geq 2$.
Let $g\colon [0,1] \mapsto \mathbb{R}$ be a function defined by conditions
\[g(0)= \begin{cases} 
      0, & a_1>0 \\
      \frac{1}{2}, & a_1=0,
   \end{cases}
\]
$$g(a_i)=g(b_i)=\frac{1}{2}+\frac{i-1}{2(m-1)}, i=1,2,\ldots,m,$$
$$g((b_i+a_{i+1})/2)=g(b_i)-\frac{1}{2},i=1,2,\ldots,m-1$$
and 
\[
  g(1)= \begin{cases} 
      \frac{1}{2}, & b_m<1 \\
      1, & b_m=1
   \end{cases}
 \]
 and whose values are given by linear interpolation between the
 defined points on the rest of the interval $[0,1]$. 
\end{definition}

\begin{figure}
  \begin{subfigure}[b]{0.32\textwidth}
    \includegraphics[width=\textwidth]{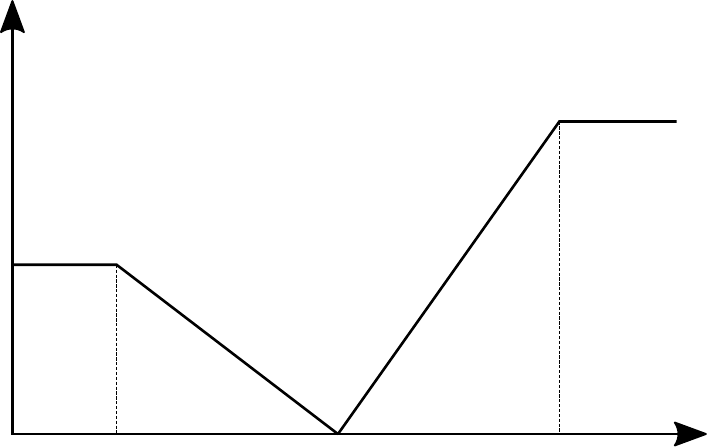}
    \caption{}
    \label{fig:1}
  \end{subfigure}
  \begin{subfigure}[b]{0.32\textwidth}
    \includegraphics[width=\textwidth]{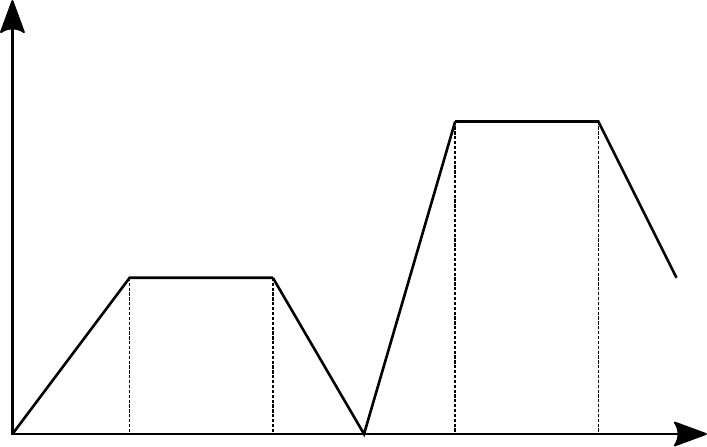}
    \caption{}
    \label{fig:2}
  \end{subfigure}
  \begin{subfigure}[b]{0.32\textwidth}
    \includegraphics[width=\textwidth]{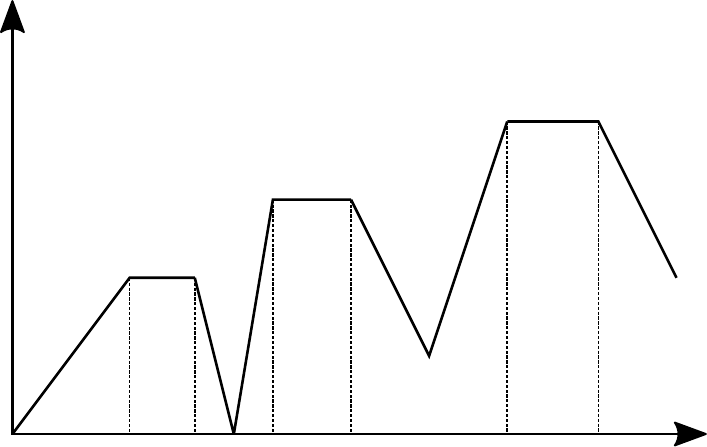}
    \caption{}
    \label{fig:3}
  \end{subfigure}
  \caption{Graphs of function $g$ for different test scenarios. On the
    left, $g$ allows buffers when testing for asymptotic independence.Values at end points differ in order to
    avoid zero variance in \eqref{thm3eq1} under asymptotic
    independence. The middle $g$ could be used to test if the
    asymptotic support  {consists of} two intervals and 
    similar $g$ could arise when testing if the support is covered by
    a single interval after the sample is processed using the method
    described in Remark \ref{asremark3}. The  $g$ {on the right}
could test if data was consistent with
    the support consisting of three  intervals and such a 
    dependence structure might arise in the search for hidden
    regular variation after the first order cone is removed from
    data.} 
  \label{fig:gexample}
\end{figure}

The function $g$ is designed so that
the user can add small buffers containing the
support.  The feature is added because in  our experience, 
it is difficult to detect asymptotic independence in real data
and it is easier if one tests for the support being in a
  bigger set.
 Such support structures still convey useful
information because they imply that some of the components
can not yield large values at the same time which is precisely the
needed information in many applications. Similar approaches for
finding sufficiently independent groups of variables exist in the
literature, for example in \cite{MR3698112}.  

The most frequently searched extremal dependence structures
correspond to asymptotic independence and strong asymptotic
dependence \cite{MR3737388}. Tests for these are presented in Remarks
\ref{asremark2} 
and \ref{asremark3} below. We first prove  a more general result from
which the others follow. The results are formulated for positive
vectors for notational simplicity.  

\begin{theorem}[Asymptotic normality of test
  statistic]\label{astheorem} Let ${\bf Z}_1,{\bf Z}_2,\ldots$ be
iid MRV random vectors in $\mathbb{R}_+^2$. Suppose
  $(R_i,\theta_i)\in \mathbb{R}_+\times C^2_+$ is the polar coordinate
  representation of ${\bf Z}_i$, where $R_i=||{\bf Z}_i||$ and
  $\theta_i={\bf Z}_i/||{\bf Z}_i||$. Let
  $\theta_{(i:n)}=(\theta_{(i:n)}^1,\theta_{(i:n)}^2)=(\theta_{(i:n)}^1,1-\theta_{(i:n)}^1)$
  be the angular component of the $i$th largest vector in $L_1$ norm
  out of a sample whose size is $n$. Assume $m\geq 2$, $g$ is as
  in Definition \ref{gdef} and  $\s_1$ is the probability
  measure induced on 
  $[0,1]$ from the angular measure
  $\s$ via the mapping $(x,y)\mapsto x$. Finally, assume $\s_1(\cup_i
  [a_i,b_i])=1$.  

Denote 
$$\mu_g:=\int_0^1 \, g(x) \, \s_1({\rm d}x)=\sum_{i=1}^m \left(\frac{1}{2}+\frac{i-1}{2(m-1)}\right)\s_1([a_i,b_i]) $$ 
and
$$\sigma^2_g:=\int_0^1 \, (g(x)-\mu_g)^2 \, \s_1({\rm d}x)=\sum_{i=1}^m \left(\frac{1}{2}+\frac{i-1}{2(m-1)}-\mu_g \right)^2\s_1([a_i,b_i]). $$ 
If  
\begin{equation}\label{asthmaseq1}
\sqrt{k(n)}\left( \ex \left(g\left(\theta_{(k(n):n)}^1\right)\right)-\mu_g \right)\to 0
\end{equation}
and $k(n)/n\to 0$, as $n \to \infty$,  then as $n\to\infty$,
\begin{equation}\label{thm3eq1}
\hat{T}:=\frac{\sum_{i=1}^{k(n)}\left(g\left(\theta_{(i:n)}^1\right) - \mu_g\right)}{\left(k(n)\right)^{\frac{1}{2}}\sigma_g}\stackrel{d}{\to}N(0,1).
\end{equation}
\end{theorem}

\begin{proof} The proof is similar to
 the proof of \cite[Theorem 3]{MR2048253}.
\end{proof}

\begin{remark}\label{asremark1}
Condition \eqref{asthmaseq1} is a second order condition
\cite{dehaan:ferreira:2006} controlling how close the asymptotic mean
is to the mean summand in the central limit theorem. It is difficult
to check in practice since detailed knowledge of $\s_1$ is not available. In addition, in practice probabilities $\s_1(
  [a_i,b_i])$, $i=1,2,\ldots,m$
may need to be estimated.\end{remark}

\begin{remark}\label{asremark2}
If $m=2$, $a_1=b_1=0$ and $a_2=b_2=1$ in Theorem \ref{astheorem}, then
\eqref{thm3eq1} is a test statistic for asymptotic independence. 
\end{remark} 

\begin{remark}\label{asremark3}
If the limit angular measure $\s$  concentrates on an interval
$[a,b]\subsetneq [0,1]$, Theorem \ref{astheorem} can not be directly
applied because  the limit distribution  must have a non-zero
variance. 
However, the case where the asymptotic support is an
interval can be reduced to the setting of two intervals by first
transforming the sample $(Z_i^{(1)},Z_i^{(2)})_{i=1}^n$. 
Assume the sample size $n$ is even. (If it is not, leave
out the observation with the smallest $L_1$ norm, because it has no 
effect on  subsequent analysis.)
When $i$ is odd, transform the two
dimensional data using mapping $(x,y) \mapsto (x/2,x/2+y)$. If $i$ is
even, use mapping $(x,y) \mapsto (x+y/2,y/2)$ instead. Then permute
the order of observations to obtain iid MRV random vectors. The
limiting measure of the transformed sample replaces the original with
two smaller copies. In addition, $\suu(\s_1)\subset [a,b]$ if and only if the
asymptotic support of the transformed sample is covered by
$[a/2,b/2]\cup[(a+1)/2,(b+1)/2]$.  
\end{remark} 

\subsubsection{Discussion on the choice of $g$ in Definition \ref{gdef}}

The function $g$ in Definition \ref{gdef} helps identify when
sets 
$[a_1,b_1], [a_2,b_2], \ldots,$  $[a_m,b_m],$
 called the {\em test
   intervals}, eventually cover the support $\suu(\s_1)$.
 Selecting a function with best performance in terms of a
pre-set benchmark depends on the rate convergence to the limit
measure and in practice, such information is not 
 available. Our suggestion for $g$ is based on experience.

There are multiple ways to define such functions, but an
asymptotic normality
  result corresponding to Proposition
\ref{astheorem} imposes requirements.
Function $g$ should be a
constant value on all separate intervals that are believed to contain
probability mass of $\s$ and $g$ must not give zero asymptotic
variance.
This rules out functions with
identical values at both endpoints of $[0,1]$.

The remaining question is how  $g$ should behave
between the regions of constant value.  We want  $g$ to
separate desirable distributions from the ones with support
that is not concentrated on the test intervals. A way to do this is to
make the quantity $|\hat{T}|$ in \eqref{thm3eq1} as large as possible
in the presence of unwanted limiting behavior. On the other hand, the
thresholded data may contain pre-asymptotic observations whose projections
are not in $\suu(\s_1)$ even when all limiting probability mass of $\s_1$ is
in
the test intervals. So, observations close to the regions
of constant value should not change the value of $|\hat{T}|$ 
dramatically. Thus, the choice of $g$ in Definition
\ref{gdef} seems reasonable and making $g$ piece-wise linear
is done  for computational simplicity.

\section{Examples with simulated and real data}\label{realdatatestsection}

In this section, we illustrate how the theoretical results concerning
support estimation in Section \ref{supportestimationsection} and
support testing in Section \ref{asindsection} can be used in
practice. We begin with a simulated dataset in Example
\ref{simulatedexample} to show how the grid based support estimator
performs in a controlled environment. Example \ref{stockexample}
studies daily stock returns. The emphasis is on the fact that stocks
in the same field tend to be dependent, but one can find at least
asymptotically independent assets among  the ordinarily listed
securities. In Example \ref{fmiexample}, 
a natural scenario for emergence of asymptotic independence is given using rainfall data\footnote{Special thanks are due to Sebastian Engelke who suggested that asymptotic independence could be found from rainfall data in a personal communication.}. Finally, in Example \ref{goldsilverexample} daily returns of gold and silver are used to show how the support estimates can be used to obtain inequalities for sizes of large fluctuations.

Typically, multivariate datasets require some amount of processing
before they can reasonably be thought to satisfy the
assumptions of multivariate regular variation given in Definition
\ref{mrvdef}. In particular, tail indices of marginal distrtibutions
must be the same for the asymptotic theory to work. To this end, one
needs to estimate tail indices. Estimation of tail index is a
classical topic which is discussed e.g. in
\cite{MR2271424,MR3191984,MR2563829} or more recently in
\cite{KIM201560}.

If  the marginals do not have the
same index, then the data needs to be processed before
proceeding further so that marginals are asymptotically equivalent.
Several methods exist for standardizing datasets to fit the scope of
multivariate regular variation  including power
transformations of marginals or the rank transform; see \cite[Section
~9.2]{MR2271424} and \cite{heffernan:resnick:2005}.  
 
\subsection{Simulated data}\label{simulatedexample}

We begin by applying the support estimator of Section
\ref{supportestimationsection} 
to simulated data. The data set consists of 3-dimensional observations
${\bf Z}_1,{\bf Z}_2,\ldots,{\bf Z}_n$, where $n=150,000$. Observations
are generated by fixing a region $B\subset C^3_+$ and then sampling
uniformly $50,000$ samples from $B$. The samples on the simplex are
then assigned an independent radial component. The radial component is
  drawn independently from the Pareto$(2)$ 
distribution. So, by definition, the angular and radial components of
the observations are independent. Additionally, $100,000$ observations
serving as noise are added to the sample by sampling uniformly from the
simplex $C^3_+$ and assigning, to each of them, an independent exponentially
distributed radial component. Finally, we put the simulated samples
into a random order so that they form an iid sample from a mixture
distribution that is MRV.

\begin{figure}[h]
  \begin{subfigure}[b]{0.24\textwidth}
    \includegraphics[width=\textwidth]{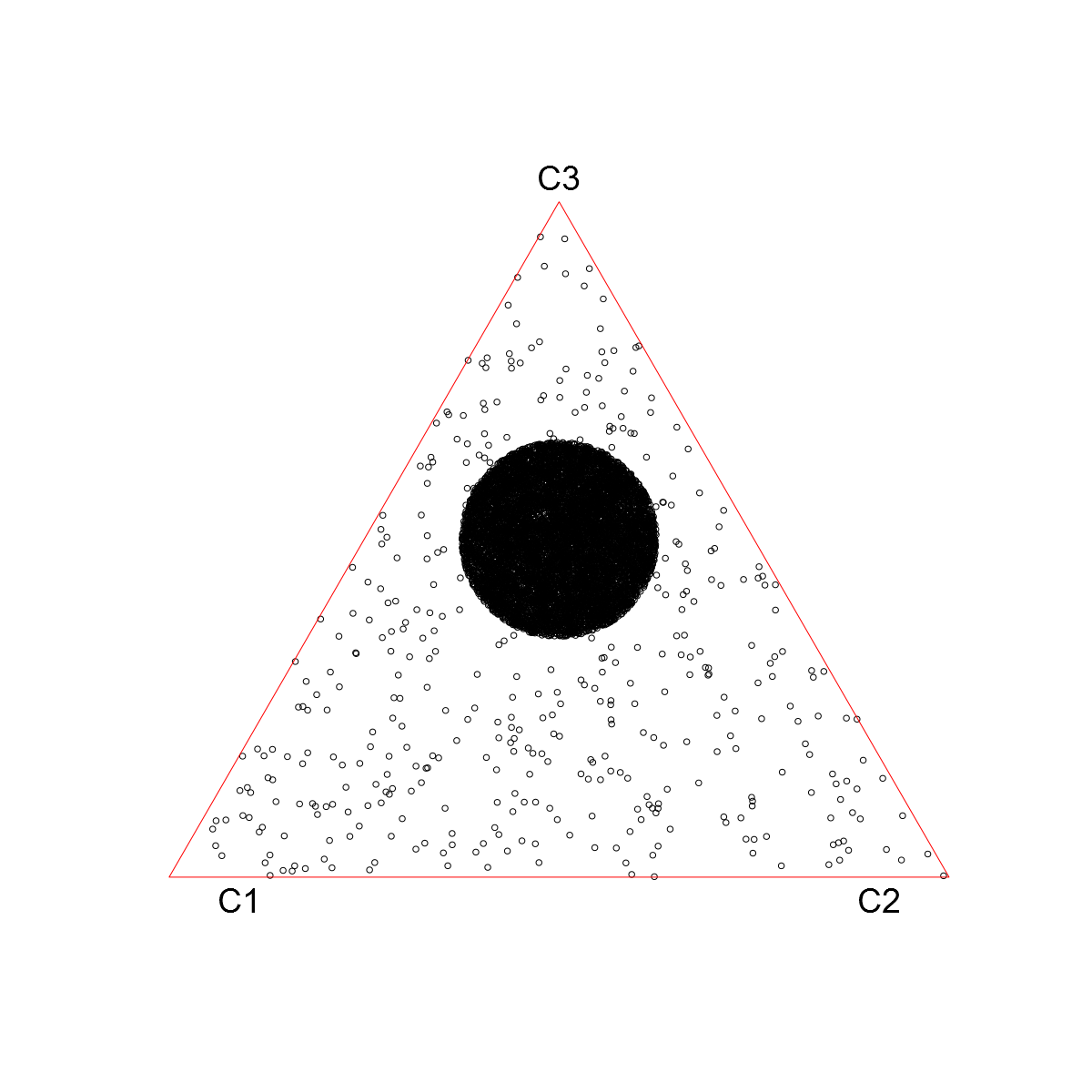}
    \caption{}
    \label{simfig:1}
  \end{subfigure}
  \begin{subfigure}[b]{0.24\textwidth}
    \includegraphics[width=\textwidth]{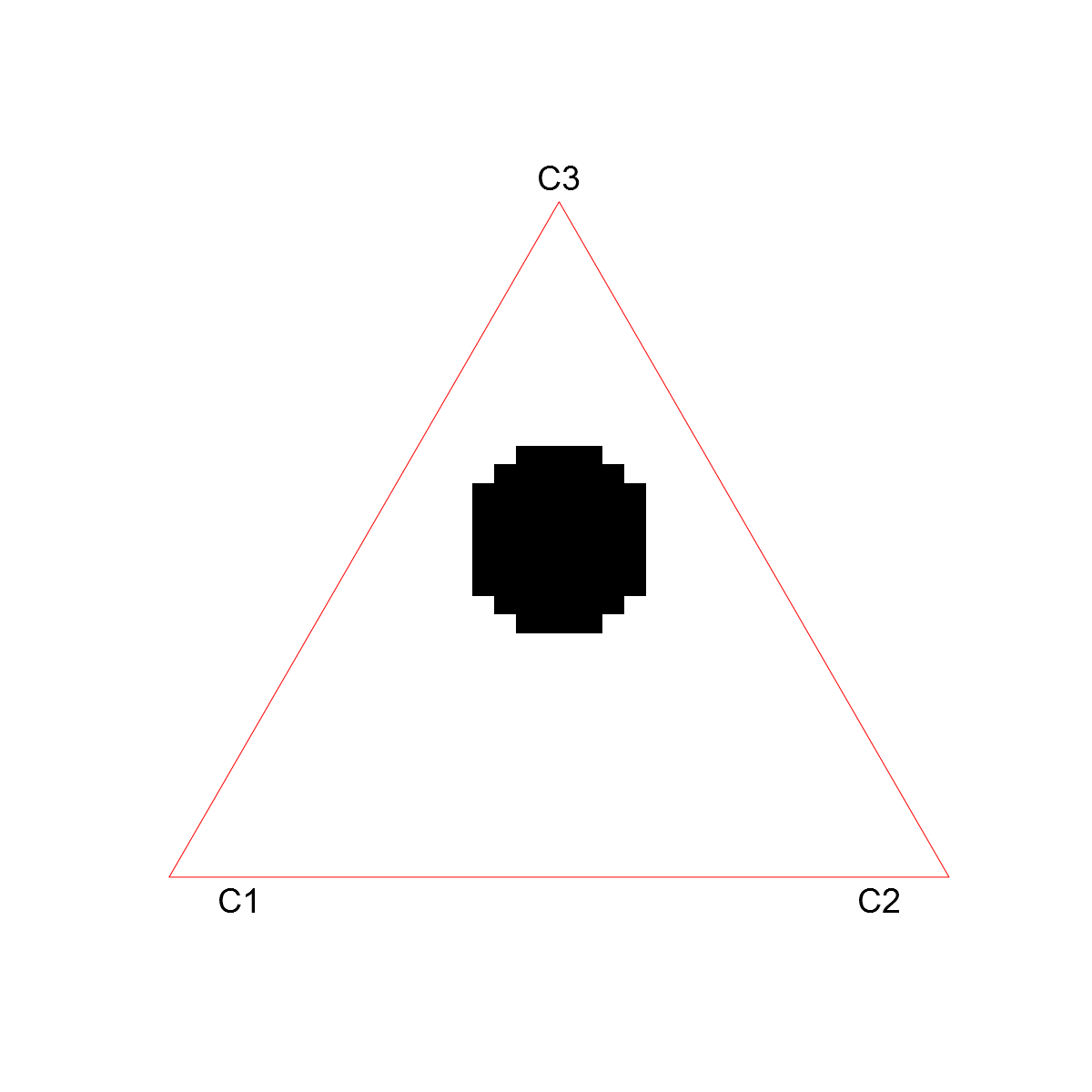}
    \caption{}
    \label{simfig:2}
  \end{subfigure}
  \begin{subfigure}[b]{0.24\textwidth}
    \includegraphics[width=\textwidth]{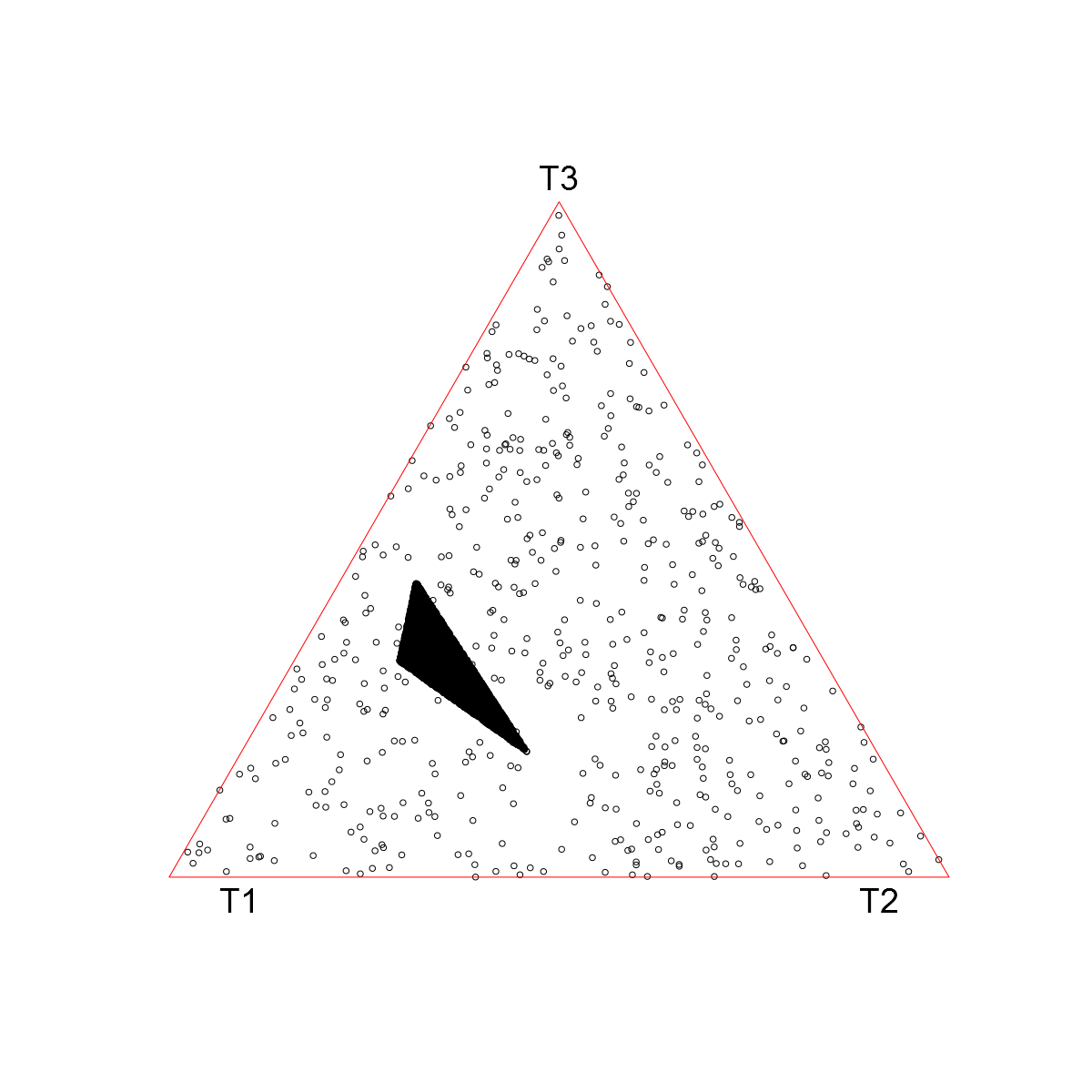}
    \caption{}
    \label{simfig:3}
  \end{subfigure}
  \begin{subfigure}[b]{0.24\textwidth}
    \includegraphics[width=\textwidth]{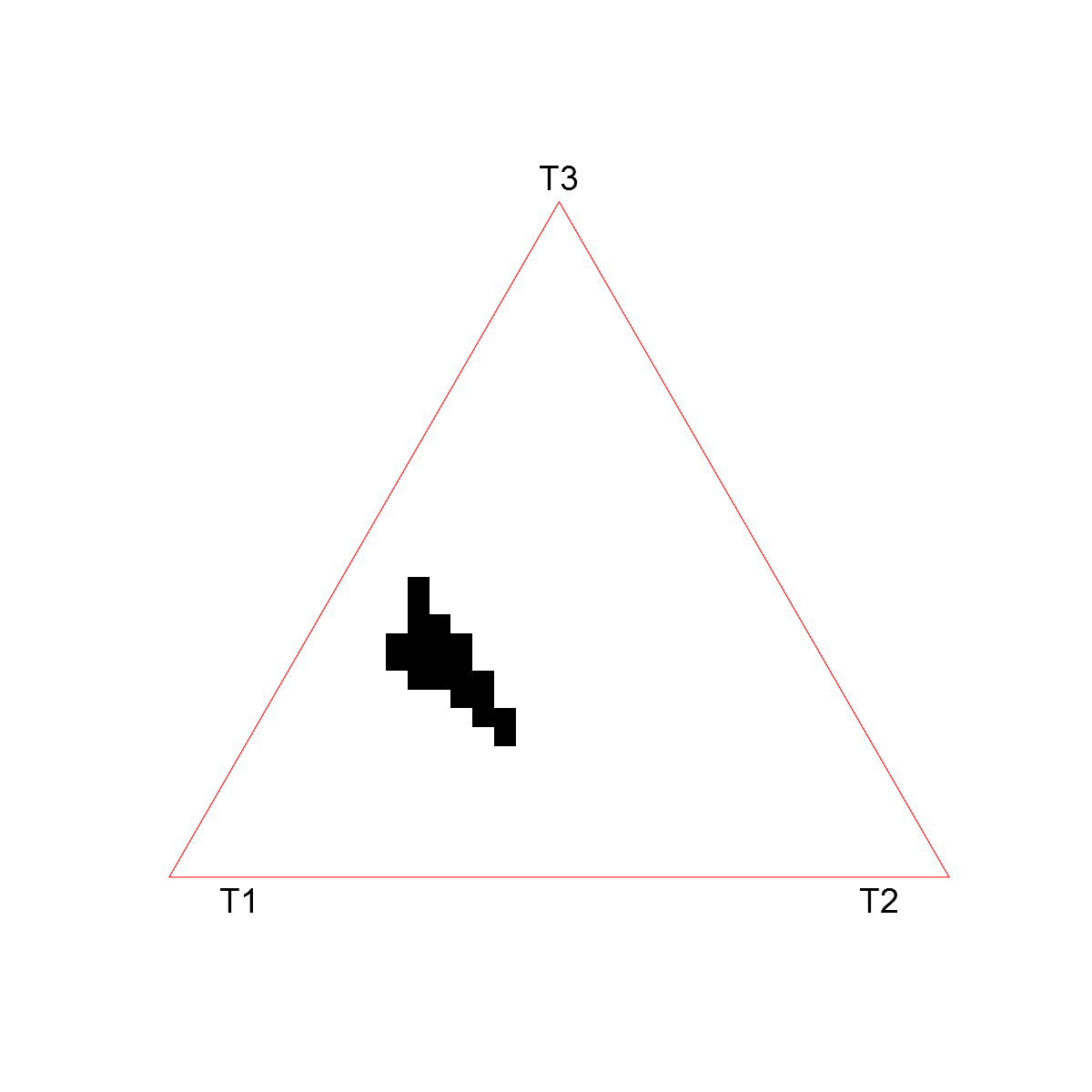}
    \caption{}
    \label{simfig:4}
  \end{subfigure}
  \caption{Figures present projected and estimated supports of simulated data. The large red triangles indicate the boundaries of the image of $C^3_+$ under a simplex mapping $T$ discussed in Section \ref{smapping} and Example \ref{exprojections}.}
  \label{fig:simexample}
\end{figure}

 Figure \ref{fig:simexample} Illustrates how well the grid based
 support estimate is able to find the location of the set $B$. The
 dots in figures \ref{simfig:1} and \ref{simfig:3} are the projected
 $k=10000$ largest observations in  $L_1$-norm. The dark dense region is the
 set $B$, which is a circle in \ref{simfig:1} and a triangle in
 \ref{simfig:3}. In figures \ref{simfig:2} and \ref{simfig:4} the set
 $B$ is estimated by forming the support estimator $\smqest$ using
 parameter values $k=10000$, $m=36$ and $q=0.01$. Rejecting 
 points with positive $q$ produces a clearly visible rasterized
 version of $B$ with no misidentified cells. This is due to the fact
 that our simulated data fits perfectly to the MRV framework.

 The
 following examples show that real data produces  less conclusive
 results. 

\subsection{Stock data vs. catastrophe fund}\label{stockexample}

In this example, we study stock market dependencies using a
data set consisting of daily prices of 6 stocks and a catastrophe
fund. The studied securities and their ticker symbol
abbreviations are: Google (GOOG), Microsoft (MSFT), Apple (AAPL),
Chevron (CVX), Exxon (XOM), British Petrol (BP) and CATCo Reinsurance
Opportunities Fund (CAT.L). Observations range from December 20, 2010
to 10 July, 2018. The data set was downloaded using the R-package {\em Quantmod}.

Observations were converted to log-returns
 by taking the logarithm of the price and calculating
differences. Apart from CAT.L
the resulting returns have similar
tail indices for positive and negative
tails.
 That is, the magnitudes of the estimated tail indices corresponding to the stock components are close to each other. However, the index of
CAT.L was substantially smaller than the others, making it necessary
to use the rank transform when comparing it against the other equities. 

\begin{figure}[h]
  \begin{center}
  \includegraphics[width=0.5\linewidth]{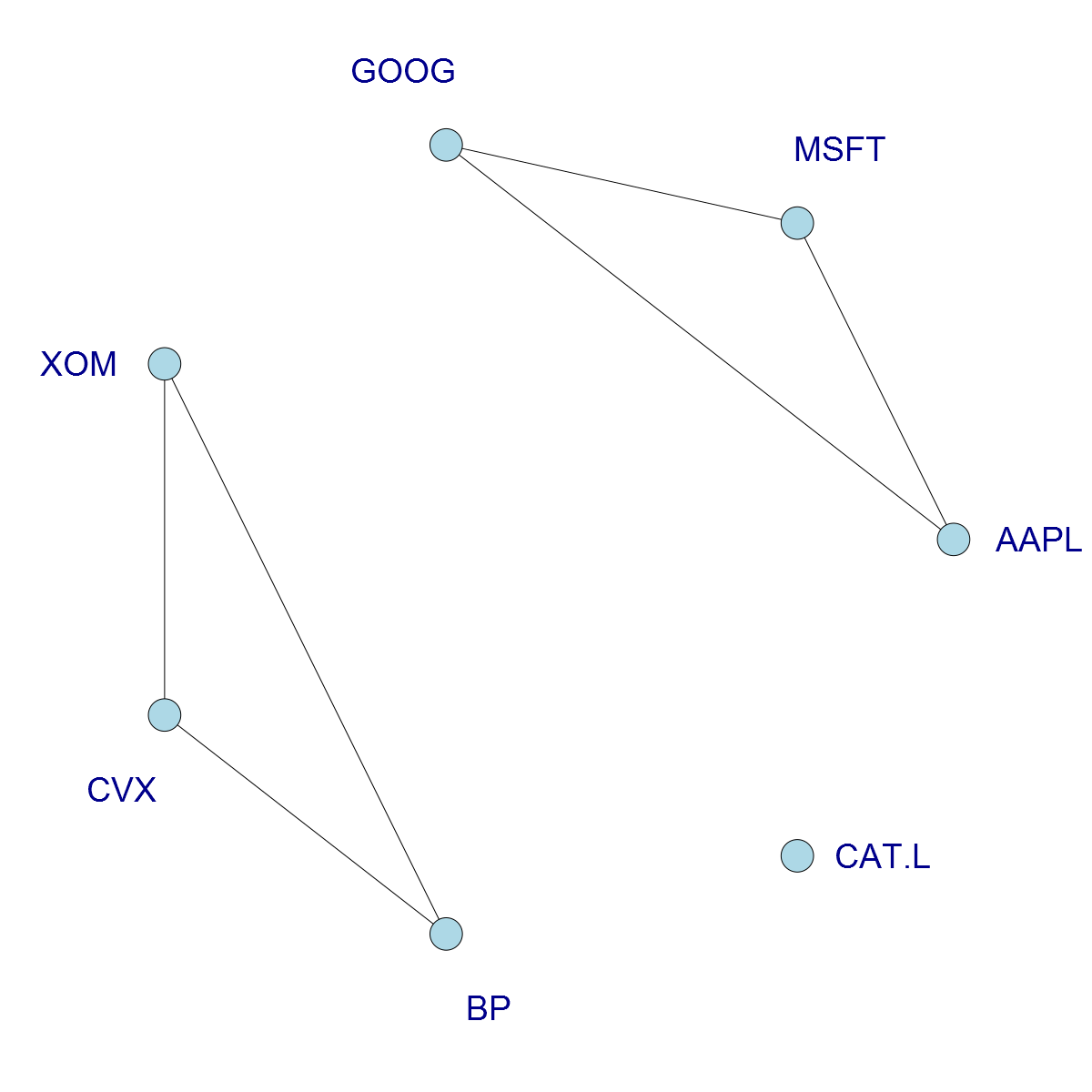}
  \end{center}
  \caption{An exploratory graph illustrating the strongest preliminary pairwise asymptotic dependencies.} 
  \label{fig:stockdatadependencies}
\end{figure}

In Figure \ref{fig:stockdatadependencies}, the strength of pairwise
dependence is calculated using the largest $k=200$ observations in
$L_1$-norm projected to $C^2_+$, denoted $z_1,z_2,\ldots,z_{200}$, by
 $$\left(\sum_{i=1}^k
(1-d_2(1/2,z_i))\right)\Big{/}k .$$ 
That is, the dependence measure assigns pairs
that have the largest
distance to the midpoint $(1/2,1/2)$ of simplex
$C^2_+$  small weights and pairs near the midpoint
large weights. All pairwise dependencies that exceed the level $0.46$
are drawn as edges in Figure \ref{fig:stockdatadependencies}. The
level was obtained empirically by gradually lowering the required
level and observing which connections appeared on the graph first,
that is, which dependencies were the strongest.  

Figure \ref{fig:stockdatadependencies} suggests that companies within
the same financial sector, oil or technology, are probably not
 asymptotically independent but that the catastrophe
fund might be asymptotically independent from stocks. 

\begin{figure}[h]
  \begin{subfigure}[b]{0.24\textwidth}
    \includegraphics[width=\textwidth]{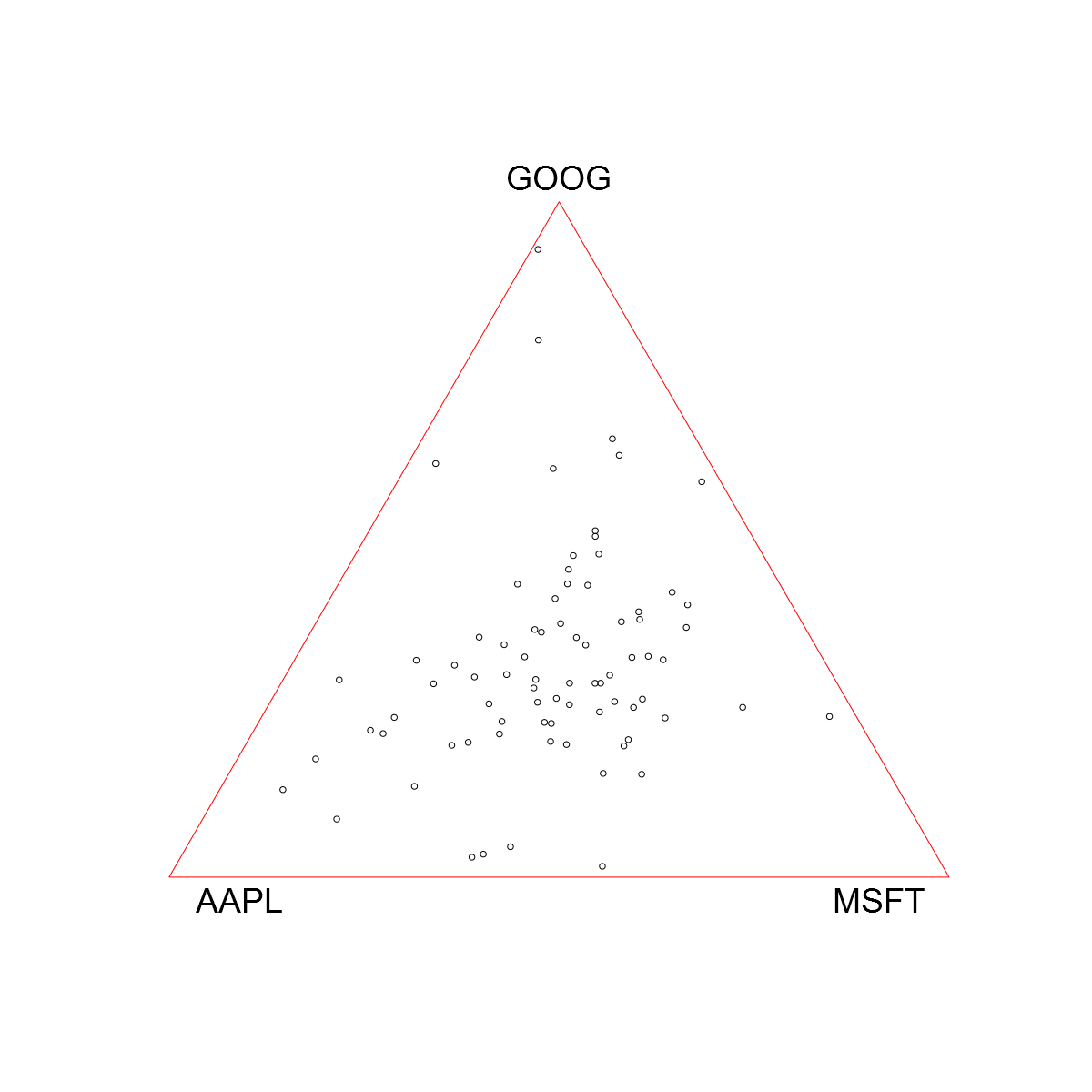}
    \caption{}
    \label{stockfig:1}
  \end{subfigure}
  \begin{subfigure}[b]{0.24\textwidth}
    \includegraphics[width=\textwidth]{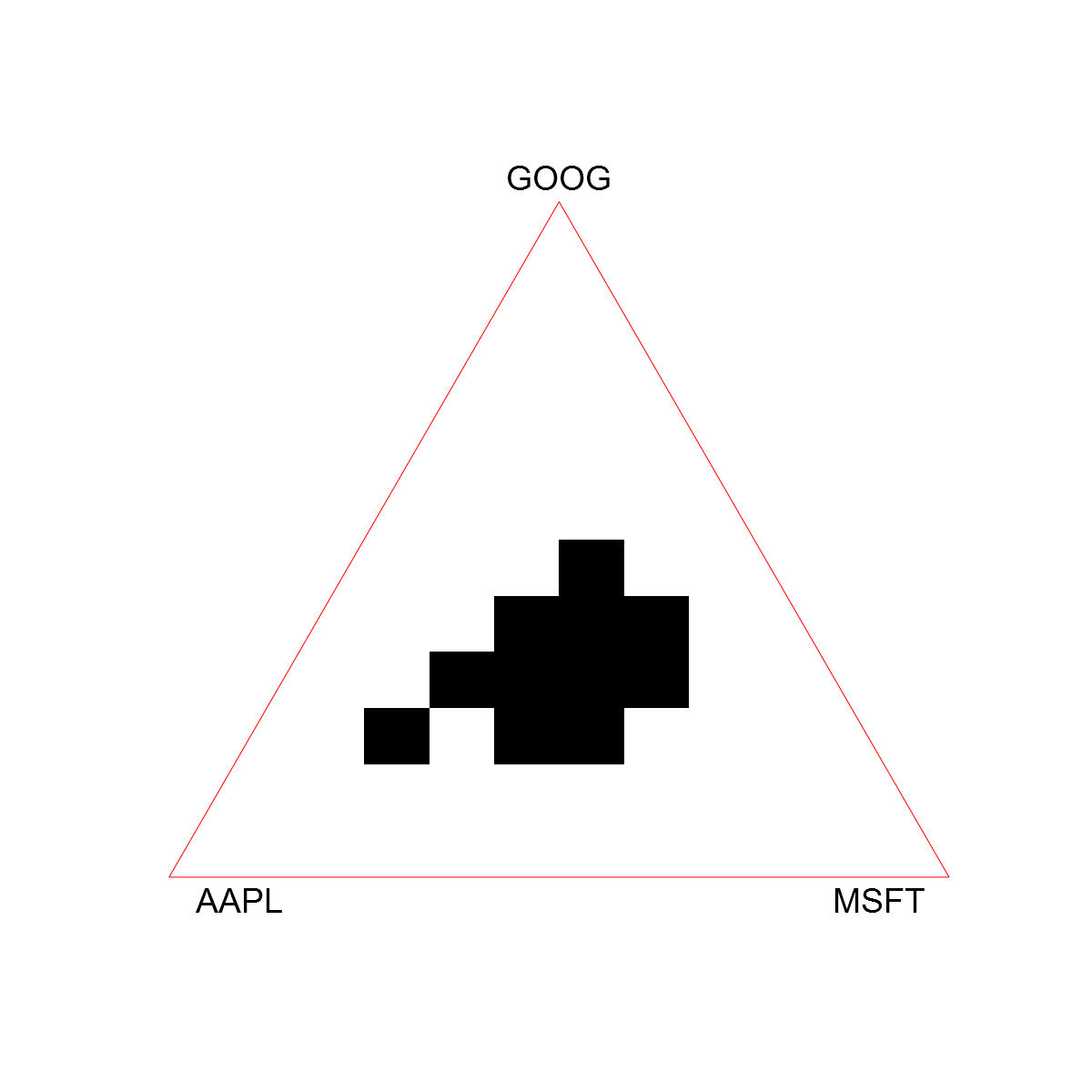}
    \caption{}
    \label{stockfig:2}
  \end{subfigure}
  \begin{subfigure}[b]{0.24\textwidth}
    \includegraphics[width=\textwidth]{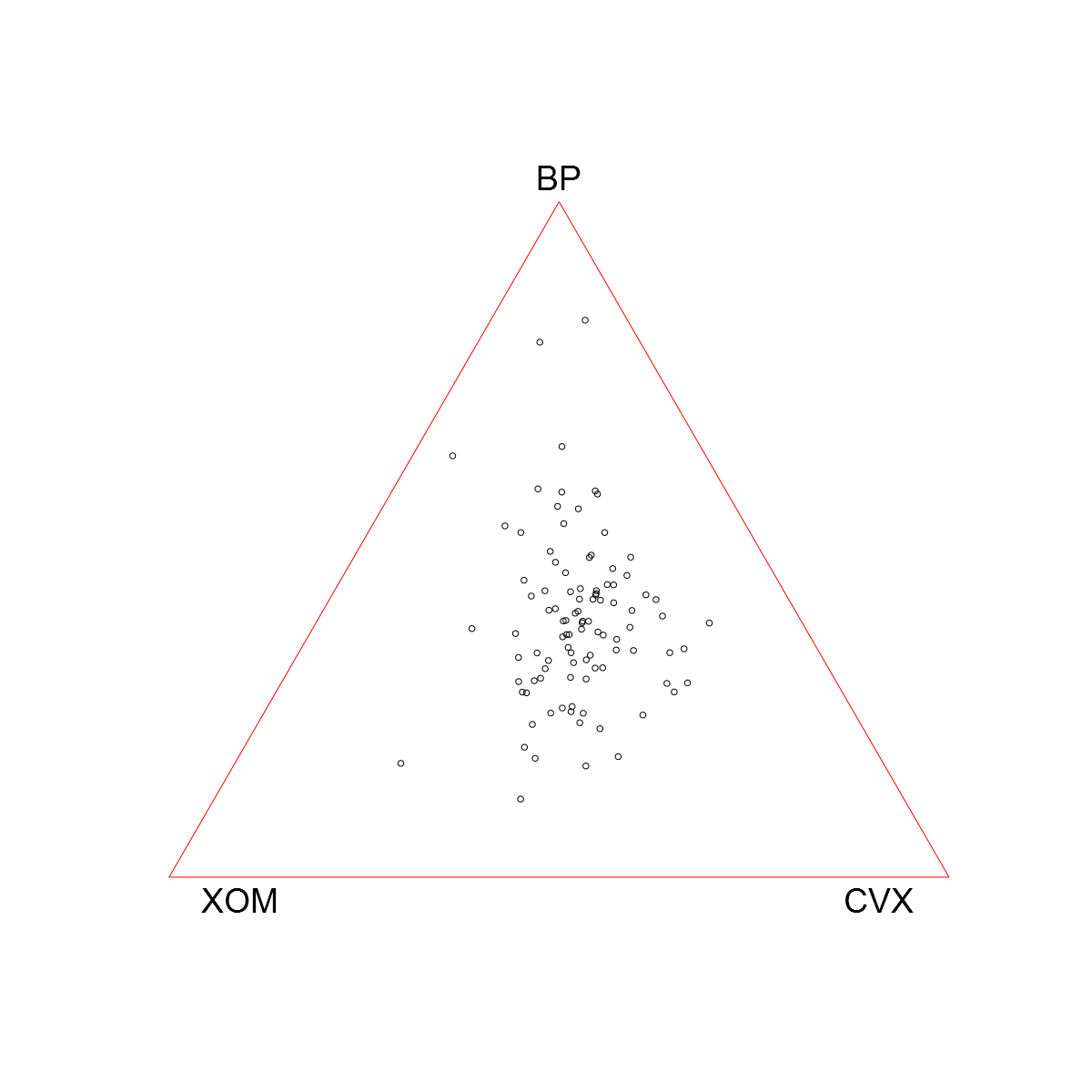}
    \caption{}
    \label{stockfig:3}
  \end{subfigure}
  \begin{subfigure}[b]{0.24\textwidth}
    \includegraphics[width=\textwidth]{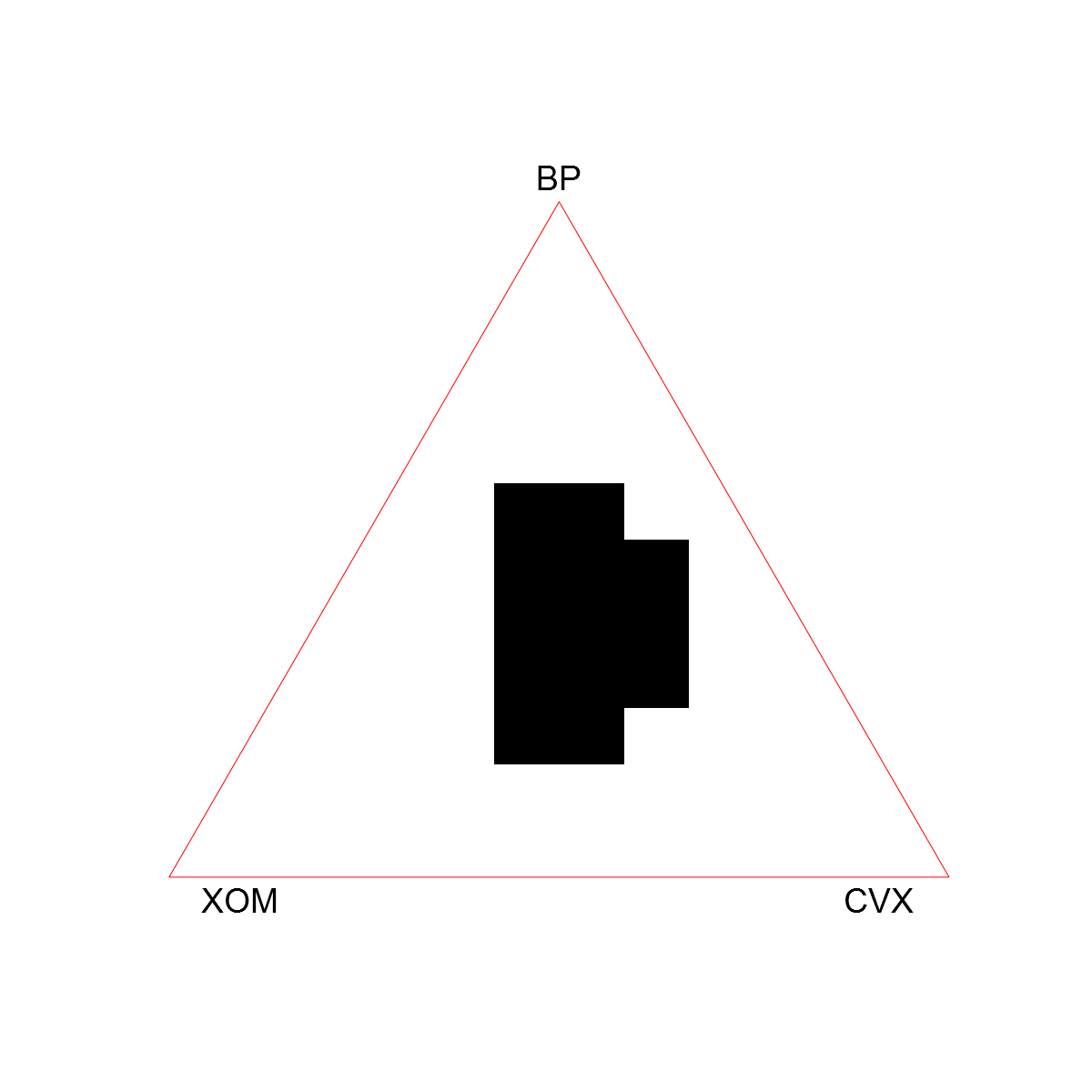}
    \caption{}
    \label{stockfig:4}
  \end{subfigure}
    \caption{Projected largest $k=200$ observations and the estimated supports for tech stocks, in Figures \ref{stockfig:1}-\ref{stockfig:2}, and for oil stocks, in Figures \ref{stockfig:3}-\ref{stockfig:4}, respectively.}
    \label{fig:stocksdependence}
\end{figure}
In Figure \ref{fig:stocksdependence}, projected and estimated supports
of positive quadrants are depicted for oil and tech stocks. We
  used parameter
values $k=200$, $m=12$ and $q=0.01$. Estimated grid based supports
in subfigures \ref{stockfig:2} and \ref{stockfig:4} suggest that
within a group, stock returns are not asymptotically independent
and, in fact, exhibit fairly strong dependence.  
However, based on Figure
\ref{fig:stockdatadependencies} oil and tech sectors might be
asymptotically independent and CAT.L could be asymptotically
independent of all the studied stocks. 

\begin{figure}[h]
    \begin{subfigure}[b]{0.24\textwidth}
    \includegraphics[width=\textwidth]{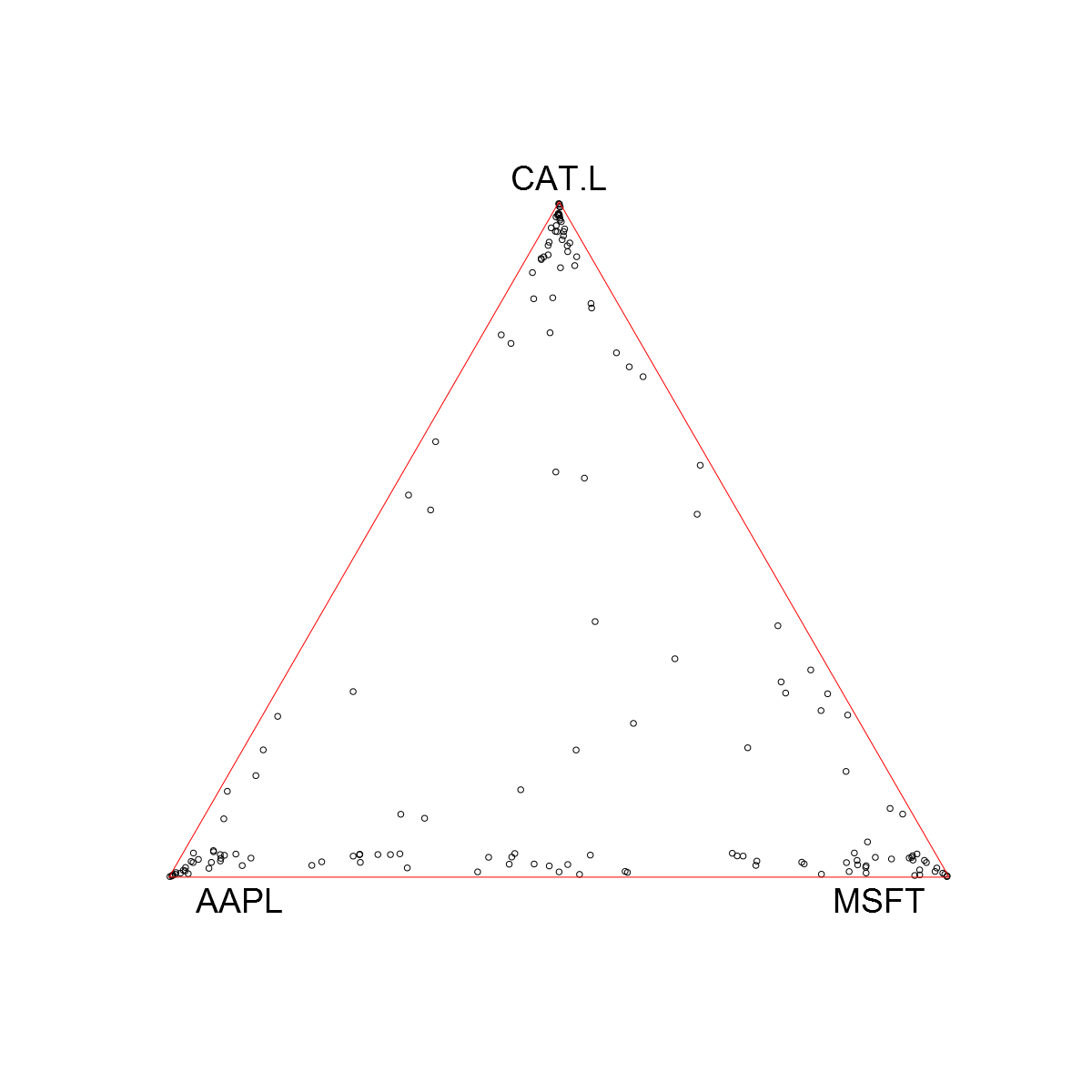}
    \caption{}
    \label{stockfig:5}
  \end{subfigure}
  \begin{subfigure}[b]{0.24\textwidth}
    \includegraphics[width=\textwidth]{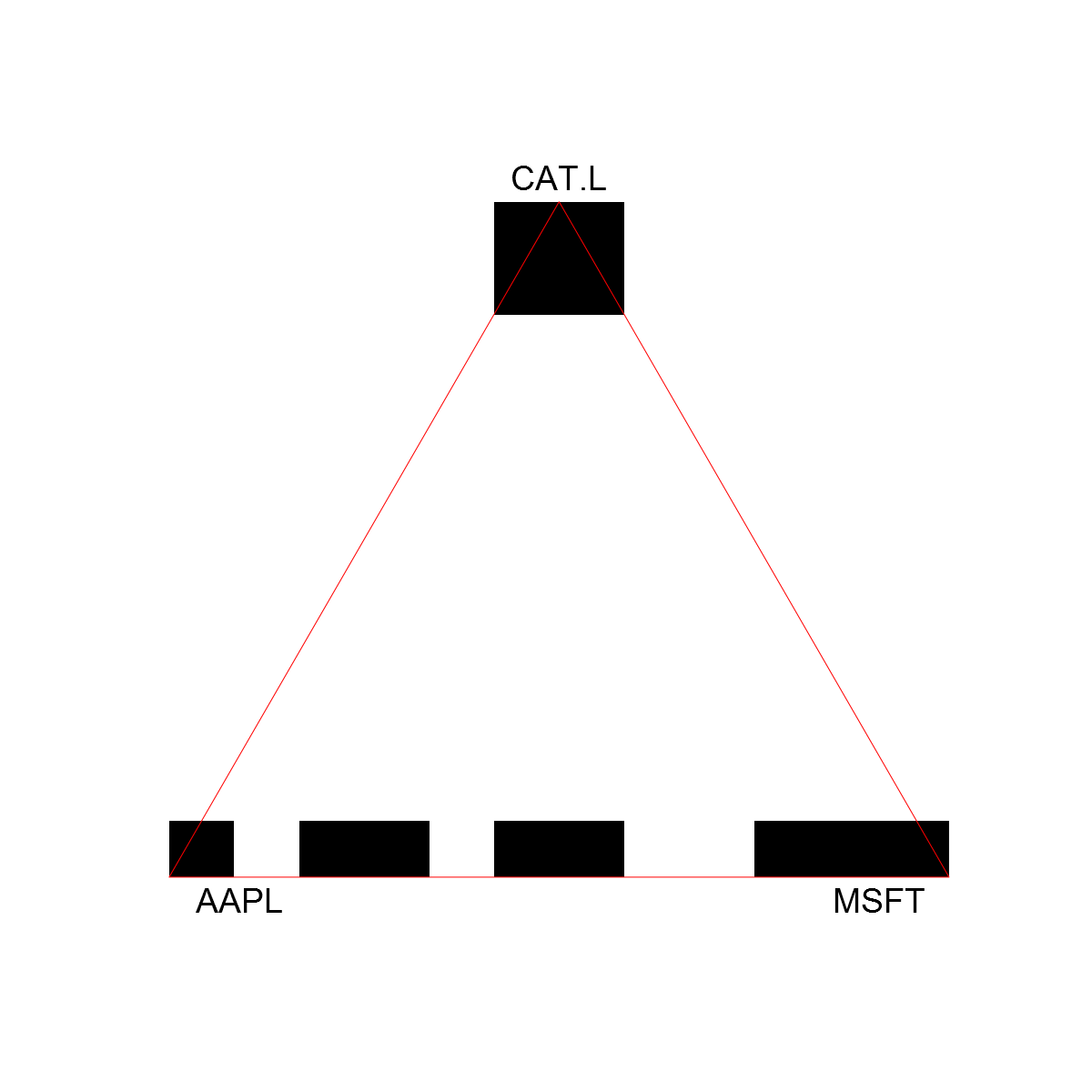}
    \caption{}
    \label{stockfig:6}
  \end{subfigure}
  \begin{subfigure}[b]{0.24\textwidth}
    \includegraphics[width=\textwidth]{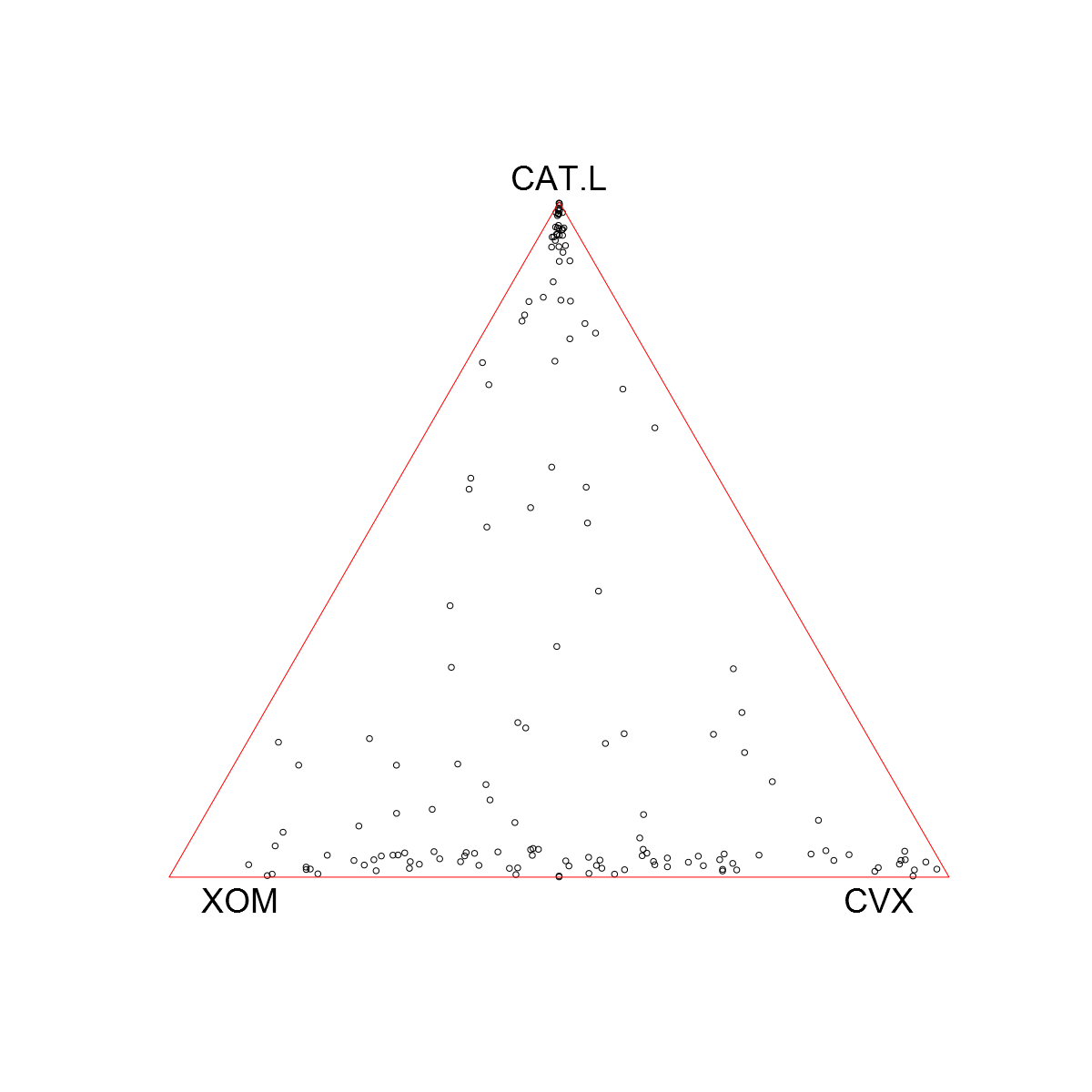} 
    \caption{}
    \label{stockfig:7}
  \end{subfigure}
  \begin{subfigure}[b]{0.24\textwidth}
    \includegraphics[width=\textwidth]{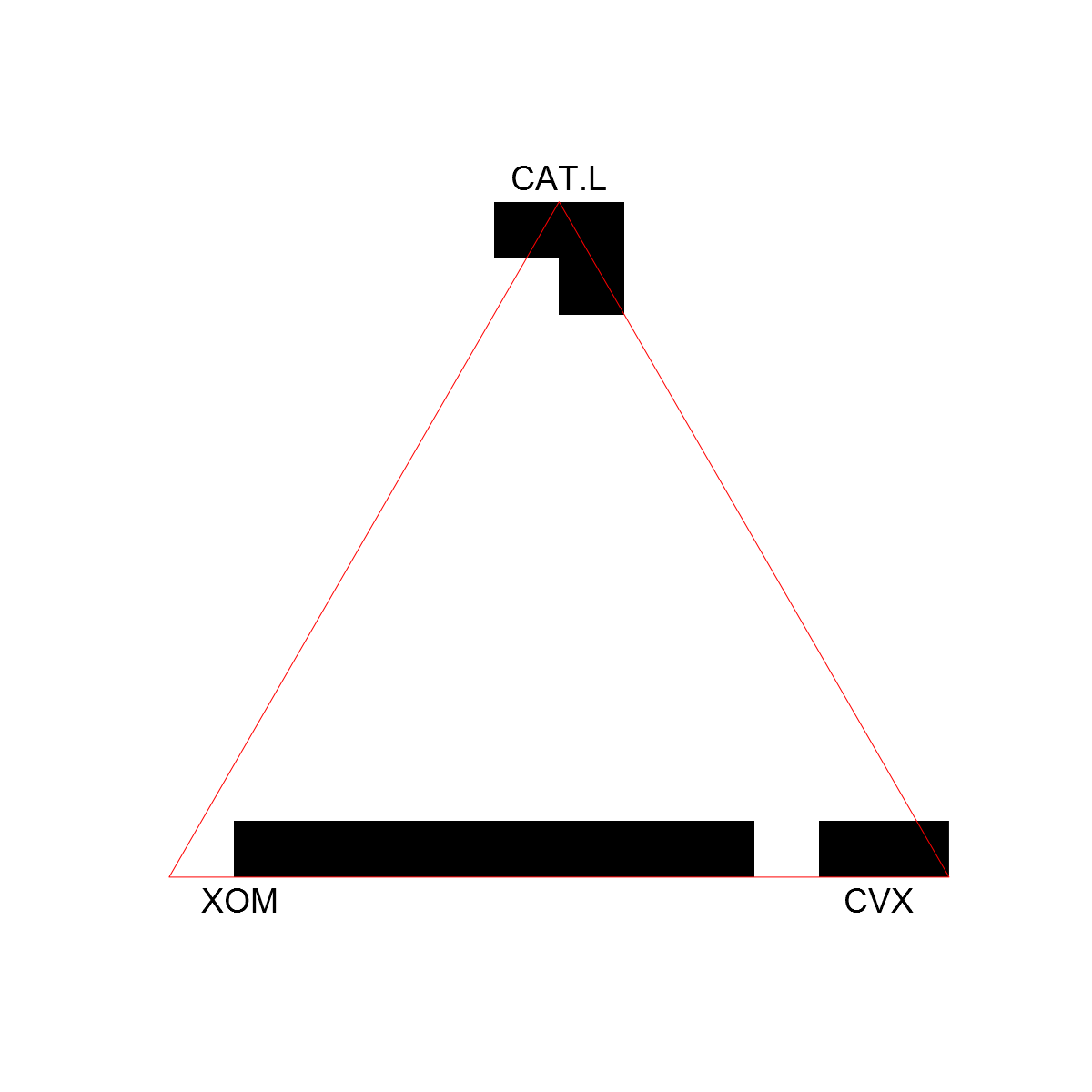}
    \caption{}
    \label{stockfig:8}
  \end{subfigure}
  \caption{Projected and estimated supports of triples of stocks
    including CAT.L.}
  \label{fig:stockexample}
\end{figure}

Asymptotic independence was tested using absolute values of
observations. Test statistics were calculated based on the $k=50$ largest observations in $L_1$ norm. Function $g$ was defined for $m=2$. The intervals of Definition \ref{gdef} were set to be of the form
$a_1=0,b_1=c,a_2=1-c$ and $b_2=1$, where $c\in \{0,0.05,0.1\}$. That is, asymptotic independence
was tested with and without buffers. In addition, it was assumed that
$\s_1([0,c])=\s_1([1-c,1])=1/2$. 
Observations from first the
oil and then the tech group were added together in order to form two dimensional
vectors. The empirical test statistic $\hat{t}$ corresponding to
$\hat{T}$ of \eqref{thm3eq1} was formed in order to calculate the probability $\prob(\hat{T}>|\hat{t}|)$ for different values of $c$. The approximate values of the probability were $2.58\cdot 10^{-5}$, $9.84\cdot 10^{-4}$ and $6.16\cdot 10^{-3}$ corresponding to $c=0$, $c=0.05$ and $c=0.1$, respectively. So, under the null hypothesis of asymptotic
  independence, there is significant evidence against asymptotic independence of the oil
and tech sectors. This is a bit surprising given the preliminary observations of Figure \ref{fig:stockdatadependencies}. 

Similar tests were
performed pairwise with CAT.L against all $6$ stocks. There was
insufficient
 evidence based on the test statistics corresponding to
\eqref{thm3eq1} to reject  asymptotic independence. More precisely, Table \ref{pairwisetable} gives approximations for the probability $\prob(\hat{T}>|\hat{t}|)$ for different values of $c$. 

\begin{center}
\begin{tabular}{ l | c c c c c c }
  c & AAPL & MSFT & GOOG & XOM & CVX & BP \\
  \hline
  $0$ & 0.23 & 0.04 & 0.10 & 0.04 & 0.08 & 0.10 \\
  $0.05$ & 0.38 & 0.13 & 0.25 & 0.18 & 0.26 & 0.26 \\
  $0.1$ & 0.40 & 0.17 & 0.30 & 0.29 & 0.34 & 0.31 \\
\end{tabular}
\captionof{table}{Values of $\prob(\hat{T}>|\hat{t}|)$ when the pairwise asymptotic independence of CAT.L and each of the $6$ stocks is studied. The column indicates which stock is tested against CAT.L. The row indicates which value of $c$ is used.}\label{pairwisetable}
\end{center}

In conclusion, the test statistic developed in Section
\ref{asindsection} supports 
the exploratory analysis of the preliminary dependence
graph in Figure \ref{fig:stockdatadependencies} in the sense that CAT.L seems to be asymptotically independent from the $6$ stocks. Asymptotic independence of the tech and oil sectors, however, was not strong enough to pass a more rigorous test.

\subsection{FMI data}\label{fmiexample}

Daily rainfall data from three separate locations
was downloaded from the Finnish Meteorological
Institute.  To reduce
seasonal effects, we only used  observations from summer
months June, July and August and 
the total number of observations   was $n=3864$.  
 Two of the locations, Kouvola and Savonlinna are
close to each other whereas the last one, Sodankylä, was further away.
Rainfall in  nearby locations showed high dependence. The
rainfall at  the further location, while not  independent of the
two others, exhibited extremal independence of the largest
observations.

\begin{figure}[h]
    \begin{subfigure}[b]{0.49\textwidth}
    \includegraphics[width=\textwidth]{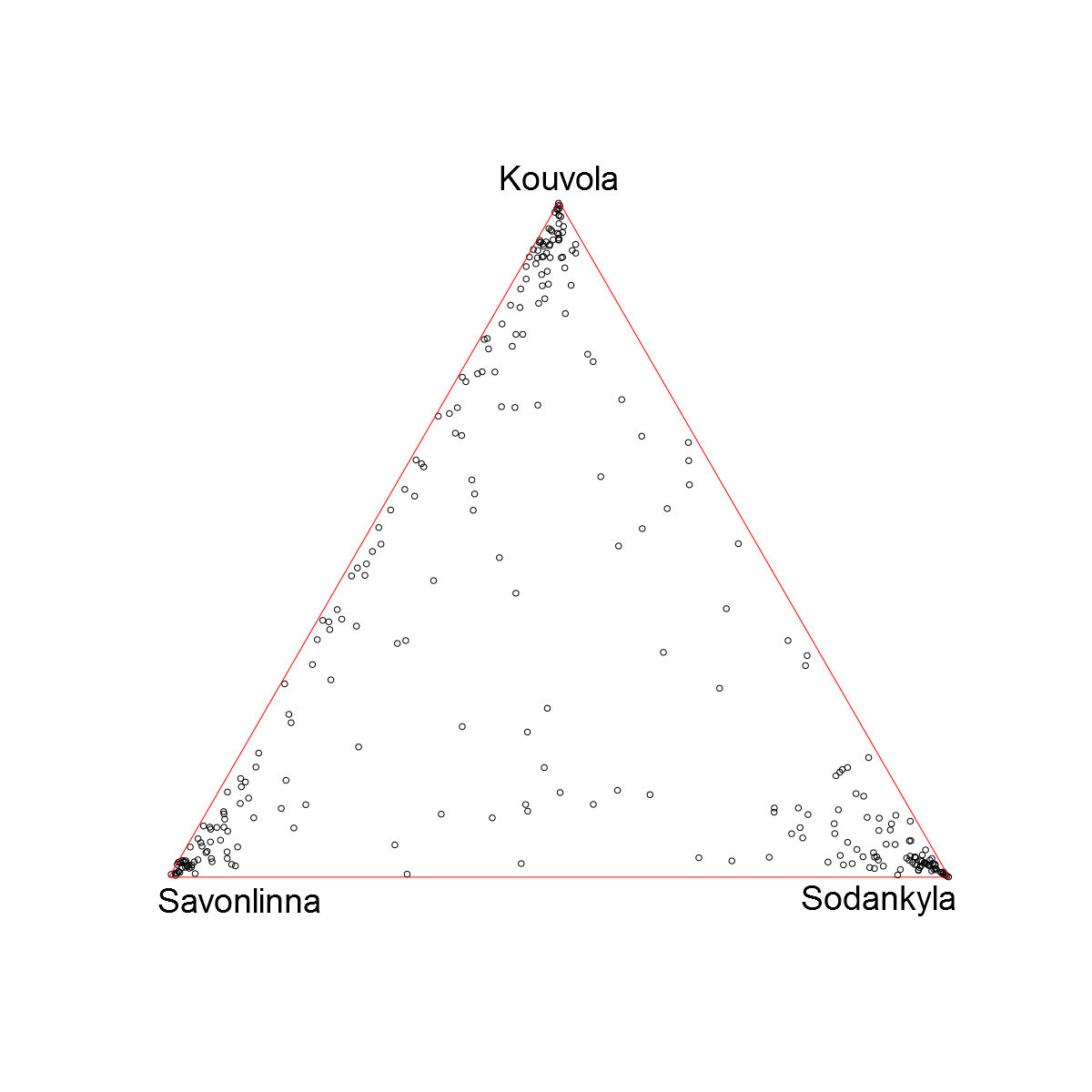}
    \caption{}
    \label{fmifig:1}
  \end{subfigure}
  \begin{subfigure}[b]{0.49\textwidth}
    \includegraphics[width=\textwidth]{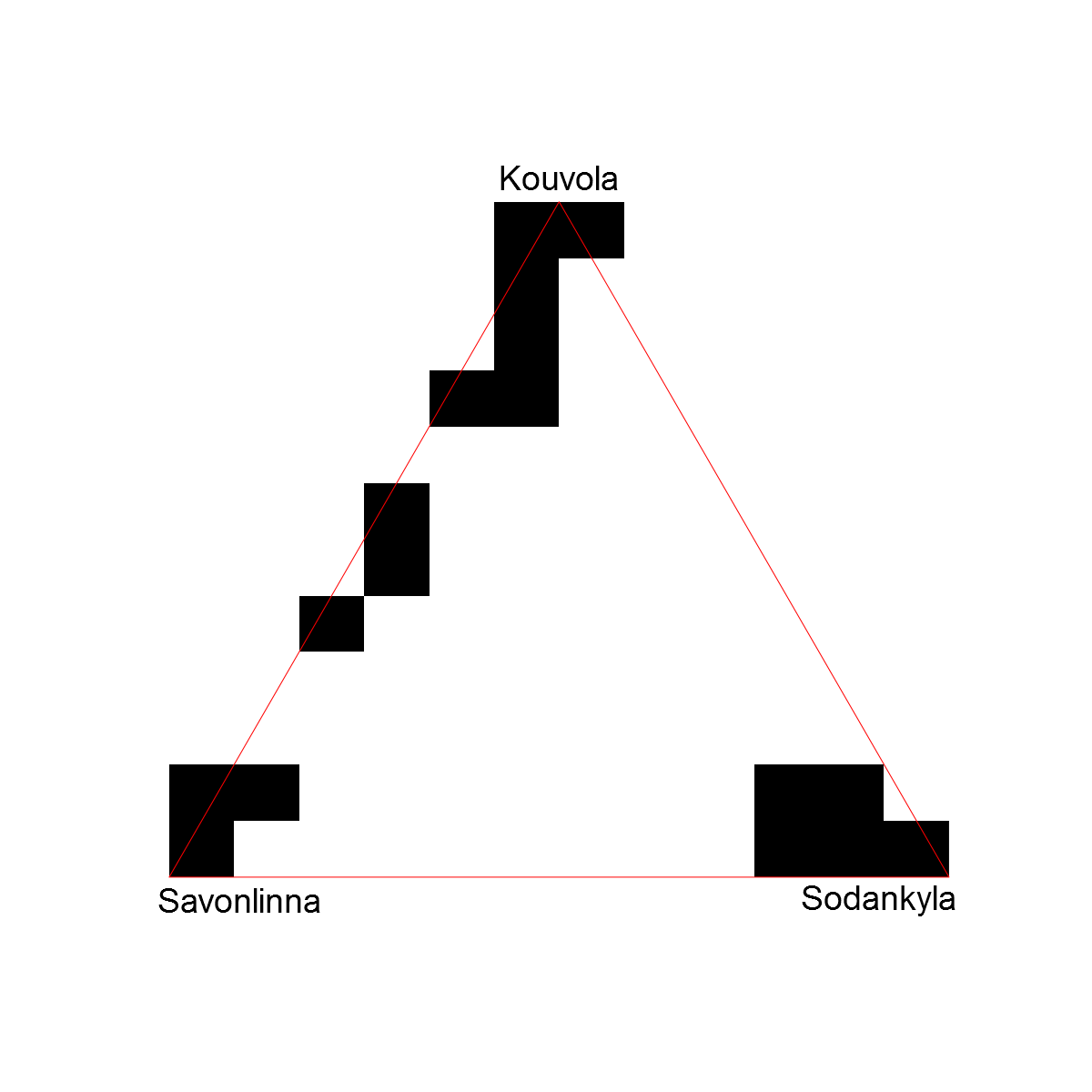}
    \caption{}
    \label{fmifig:2}
  \end{subfigure}
  \caption{Projected and estimated supports of daily rainfall data recorded in 3 locations in Finland.}
  \label{fig:fmiexample}
\end{figure}

Figure \ref{fig:fmiexample} shows projected 3-dimensional
  points and
estimated supports of the
rank transformed rainfall vectors  using $k=300$ largest
observations. For the support estimate, we used parameter values $m=12$ and
$q=0.01$.
 The rainfall data supports the idea that
locations in close proximity (Savonlinna, Kouvola) are 
dependent and locations far 
away from each other are asymptotically independent.

\subsection{Gold vs Silver price data}\label{goldsilverexample}

In this section, we study a data set consisting of daily gold
  and silver prices.  The data is gathered from London Bullion Market
Association. It is downloaded via the R-package {\em
  Quandl}. In the data, the price of one ounce of gold or silver is
recorded each day during a time period ranging from December 3, 1973
to January 15, 2014. Only complete cases where the price information
was available from both gold and silver were accepted as part of the
data set. There were three days where price information was
incomplete. Large price fluctuations did not occur during the omitted
days and thus ignoring them has no effect to the resulting asymptotic
analysis. 

We transformed the daily price data to log-returns 
to obtain a sample which is better suited with the
iid assumption of the model. The individual positive and negative
marginals of gold and silver were reasonably heavy tailed.
No power or rank transformations seemed necessary to standardize
the data set. The resulting sample of $n=10323$ was
thresholded by the $k=200$ largest observations in the $L_1$-norm and then
projected onto $C^2$ to produce the diamond plot presented in Figure
\ref{fig:goldsilverdata}.

\begin{figure}[!htb]
  \begin{center}
  \includegraphics[width=0.7\linewidth]{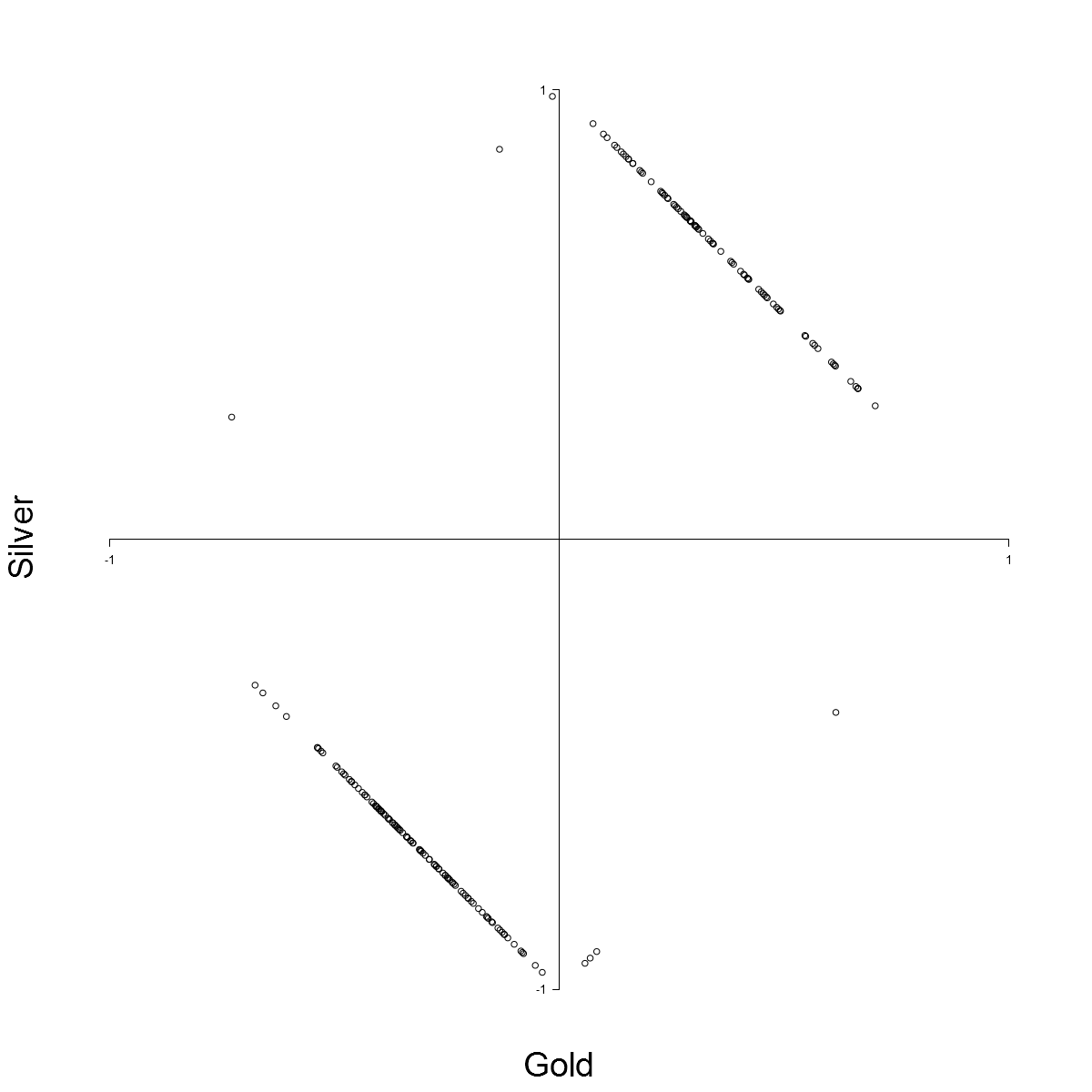}
  \end{center}
  \caption{Diamond plot of daily price returns of gold and
    silver. 
The horizontal axis corresponds to gold and
    vertical axis to silver.}
  \label{fig:goldsilverdata}
\end{figure}

Figure \ref{fig:goldsilverdata} shows that the largest fluctuations in
gold and silver prices tend to occur to the same direction. In
addition, it seems that the points do not fill the positive or
negative quadrant of the $C^2$ simplex evenly, but concentrate on
intervals. The estimation of the asymptotic support in the negative
quadrant was chosen as a suitable example, analysis of the other quadrants
could be performed similarly. So, only the part of data where both
components are negative was used. The $n=3951$ observations were
multiplied by $-1$ to obtain a data set in the positive quadrant.   

The one dimensional grid based estimator was obtained using the first
$1975$ observations sampled uniformly without replacement from the
data. The points were projected using a simplex mapping $T\colon C^2_+
\to [0,1]$ defined by $T(x,y)=x$. Since gold is on the horizontal
axis, the projected values on $[0,1]$ close to $0$ correspond to
large losses in silver and values near $1$ correspond to
  large losses in gold prices. 
  
  The grid based support estimator
with parameter values $n=1975$, $k=100$, $m=15$ and $q=0.02$ suggests
that the asymptotic support should be covered by interval $[0,0.65]$. 

\begin{figure}[h]
    \begin{subfigure}[b]{0.24\textwidth}
    \includegraphics[width=\textwidth]{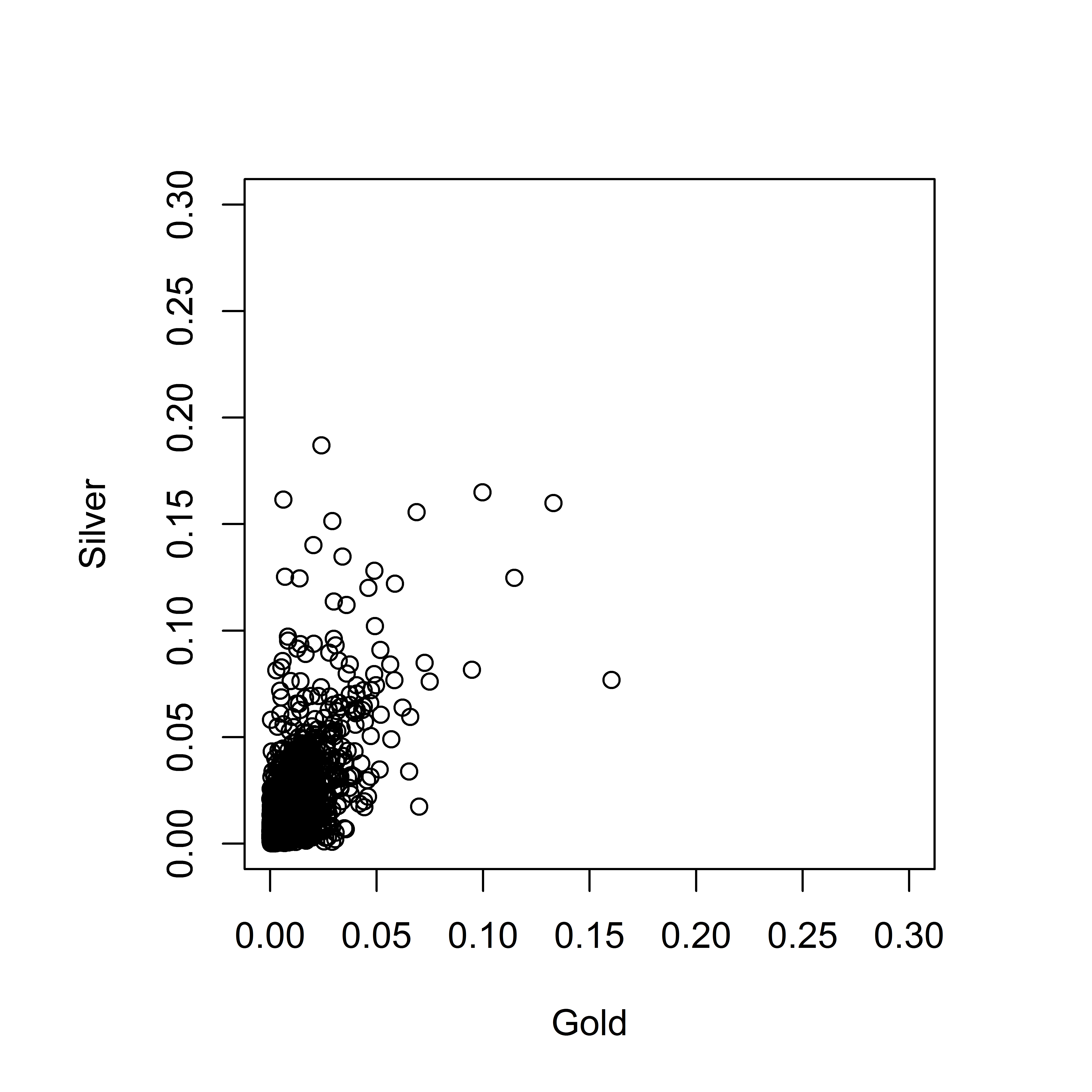}
    \caption{}
    \label{negativesfig:1}
  \end{subfigure}
  \begin{subfigure}[b]{0.24\textwidth}
    \includegraphics[width=\textwidth]{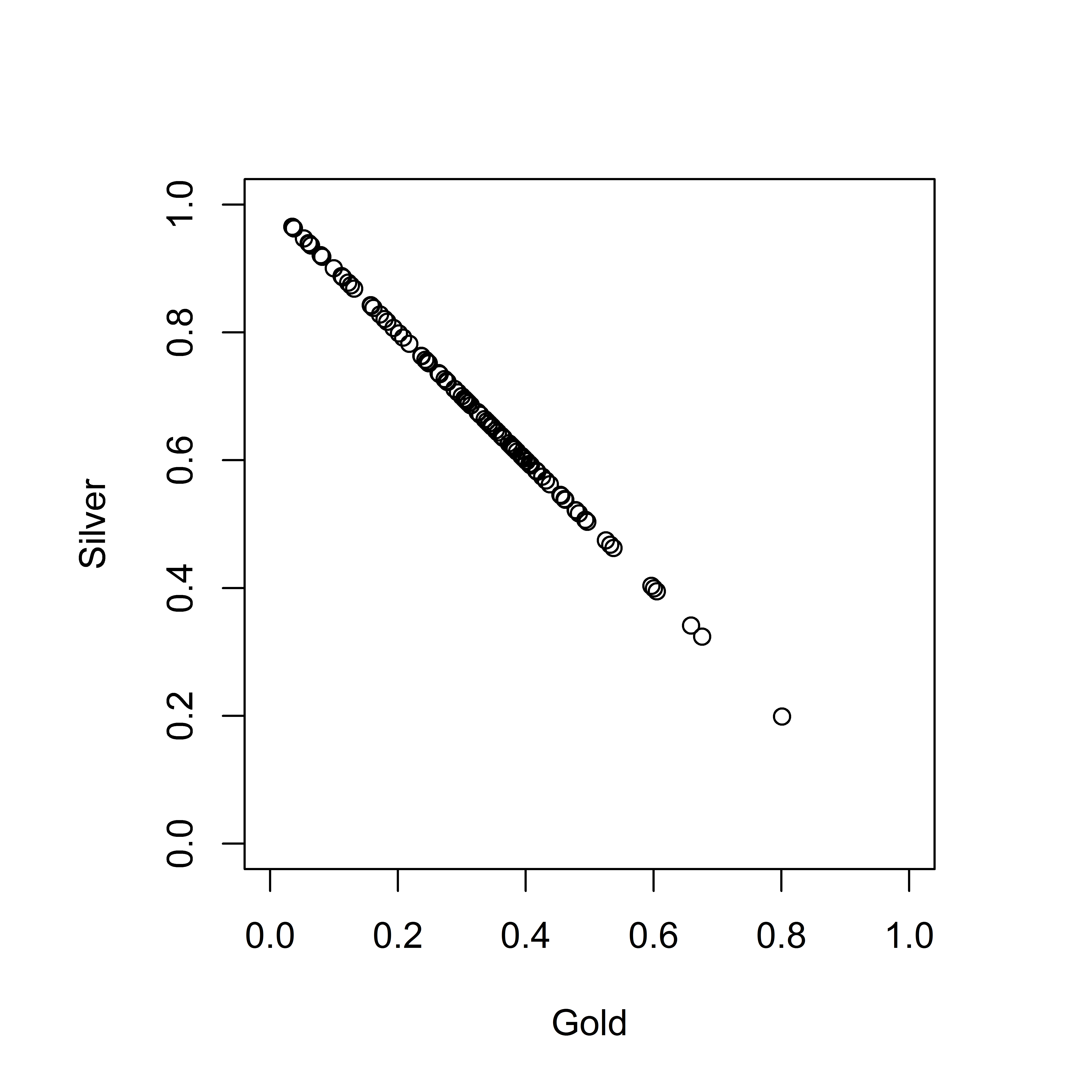}
    \caption{}
    \label{negativesfig:2}
  \end{subfigure}
    \begin{subfigure}[b]{0.24\textwidth}
    \includegraphics[width=\textwidth]{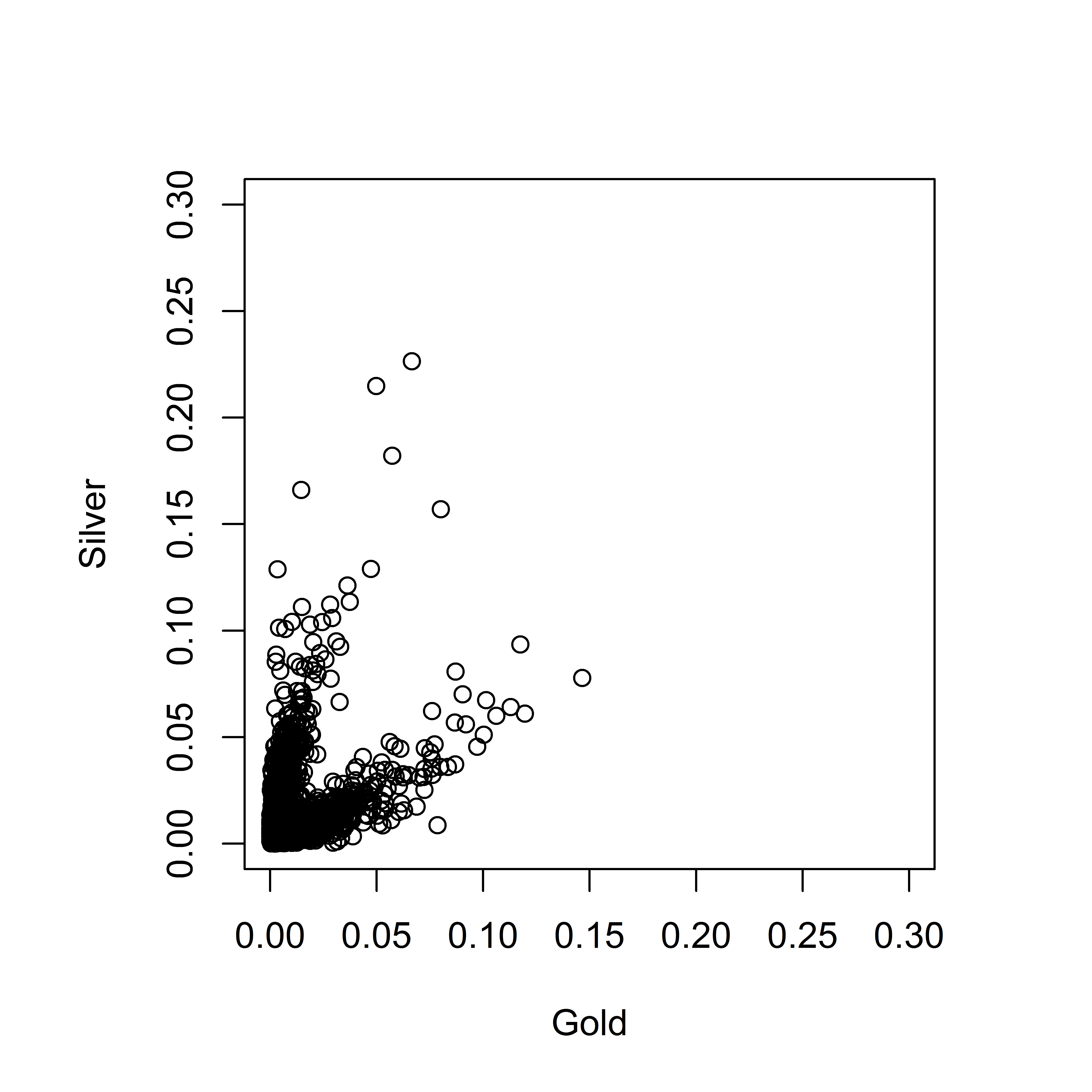}
    \caption{}
    \label{negativesfig:3}
  \end{subfigure}
    \begin{subfigure}[b]{0.24\textwidth}
    \includegraphics[width=\textwidth]{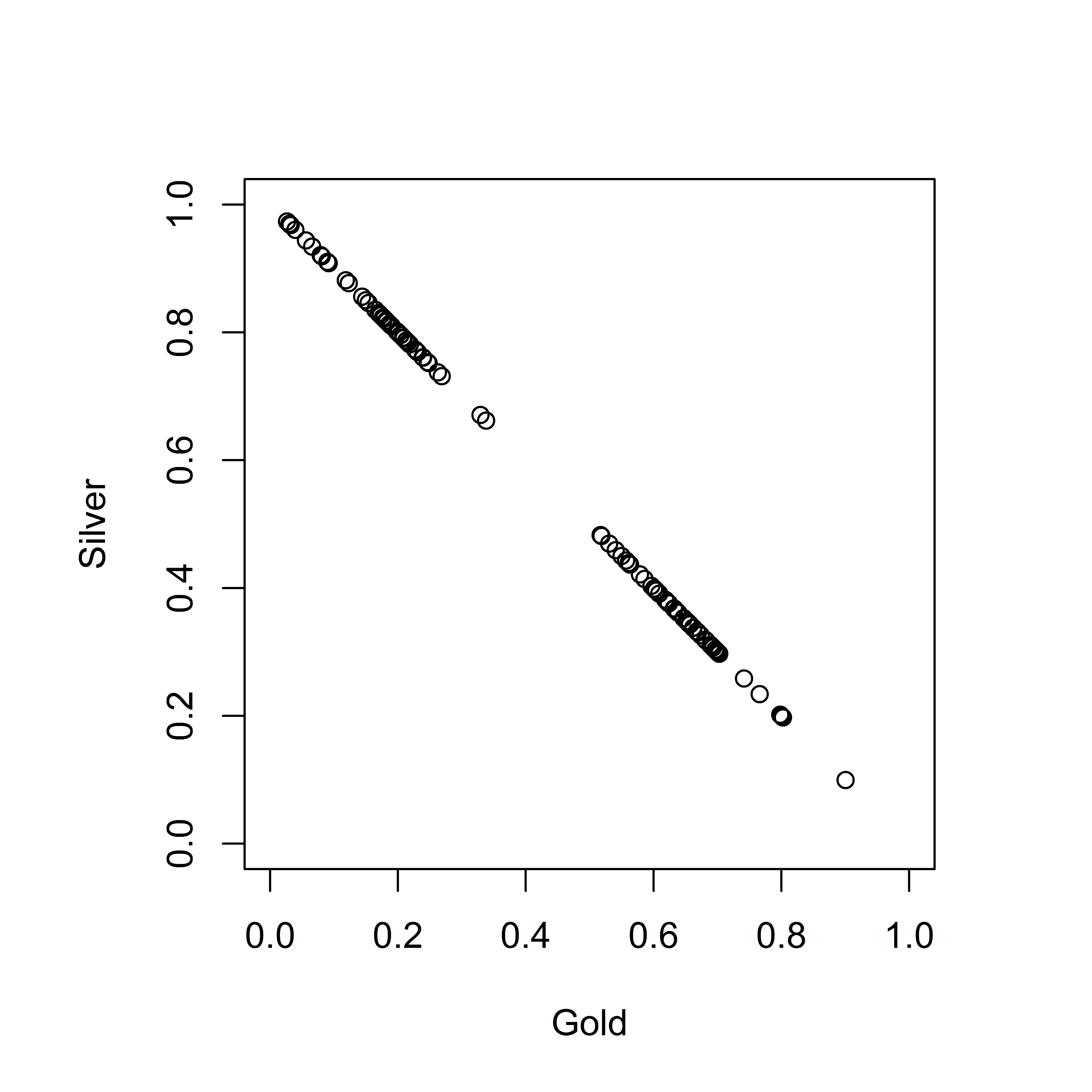}
    \caption{}
    \label{negativesfig:4}
  \end{subfigure}
  \caption{In Figure \ref{negativesfig:1}, the original data set is transformed
    using the method of Remark \ref{asremark3}. The transformed
    observations are presented in Figure \ref{negativesfig:3}. Figures
    \ref{negativesfig:2} and \ref{negativesfig:4} show the diamond
    plot of the $k=100$ largest obervations in $L_1$ norm; Figure
    \ref{negativesfig:2} corresponds to the data before application of
    the Remark \ref{asremark3} procedure and Figure
    \ref{negativesfig:4} is after the procedure.} 
  \label{fig:negativesfig}
\end{figure}

The validity of the support estimate was tested using the remaining
$1976$ observations. Function $g$ was formed using the method
of Remark \ref{asremark3}. The process is illustrated in
Figure \ref{fig:negativesfig}. The aim is to test if the asymptotic
support of the transformed data is covered by $[0,0.325]\cup
[0.5,0.825]$. The null hypothesis is that in our sample $\hat{T}\sim
N(0,1)$ where $\hat{T}$ is as in Equation \eqref{thm3eq1}. The
empirical test statistic $\hat{t}$ corresponding to quantity
$\hat{T}$
was calculated from the remaining observations with result
$\hat{t}\approx 0.074$. Under the null hypothesis
$\prob(\hat{T}>|\hat{t}|)\approx 0.398$. So, the value of $\hat{t}$
gives no reason to reject the null hypothesis or the idea that the
asymptotic support of the original sample is a subset of  $[0,0.65]$. 

As a practical application, we immediately obtain inequalities for
large fluctuations in gold and silver prices. Denote the daily
logarithmic decrease in prices with $x$ for gold and $y$ for
silver. If a very large decrease is observed for gold, i.e. $x$ is
large, then the support estimate implies $x/(x+y)\leq 0.65$ so that
$y\geq 0.53 x$. In other words, the support estimate says it is
unlikely for the decrease in logarithmic silver price to be less than
$0.53 x$. So in the presence of extremal dependence, the method allows
qualitative conclusions about
otherwise unknown quantities.

\subsection{Final thoughts}

The examples show  our methods are
usable in some scenarios but asymptotic support
estimation obviously has limitations. For one, it is challenging to find a
large sample of vectors with tail equivalent marginals that satisfies
the iid assumption. Thus, data pre-processing is required to get
  the data into usable form. With time series, larger number of observations
may lead to poor results because of lack of stationarity. 
In financial contexts, a popular pre-processing method is
{\em de-GARCHing}, see \cite[Sec 2.1.]{HOFERT2017}. The choice of
pre-processing method adds a new source of uncertainty to the model.  

Existence of an angular limit measure requires
a {\it standard\/} MRV model in which marginal tails are tail
equivalent. Theoretically, non-standard MRV models can be transformed
to standard 
and the rank transform or power transform are the data analogues of
the transform \cite{MR2271424}.
While rank transformed data can consistently estimate the limit
measure
\cite{MR2271424, heffernan:resnick:2005},
it is not clear what effect such a transform applied to finite samples
has on support estimation.

The proposed method in Sections \ref{supportestimationsection} and
\ref{asindsection} has an exploratory component since the
support estimate requires choice of $k, m, q$.
We are developing  dedicated
software that facilitates such choices and graphically shows effects
of the choice on support
identification and testing. 

\subsection*{Acknowledgements}
Jaakko Lehtomaa gratefully acknowledges hospitality and use of the
services and facilities of Cornell University's School of Operations
Research and Information Engineering 
during his visit from Sep 2017 to Aug 2018, funded by the Finnish
Cultural Foundation via the Foundations’ Post Doc Pool. Sidney Resnick
was partially supported by US ARO MURI grant W911NF-12-1-0385 to
Cornell University and final versions of this paper were written
while visiting Australian National University's College of Business
and Economics and grateful acknowledgement is made for their
hospitality and support.

\begin{footnotesize}
\bibliography{bibliography}{}
\bibliographystyle{acm}
\end{footnotesize}

\end{document}